\documentclass[11pt,a4paper,oneside]{amsart}
\usepackage{amsfonts, amsmath, amssymb, amsthm}

\usepackage[colorlinks=true,citecolor=blue]{hyperref}

\usepackage[scale=0.75]{geometry}
\usepackage{graphicx}
\usepackage{parskip}
\usepackage{enumerate}
\usepackage{verbatim}
\usepackage{tikz}

\usepackage{txfonts}
\usepackage[T1]{fontenc}

\usepackage{lineno}
\usepackage{booktabs}
\newtheorem{theorem}{Theorem}[section]
\newtheorem{lemma}[theorem]{Lemma}

\newtheorem{claim}[theorem]{Claim}
\newtheorem*{unthm}{Theorem}
\newtheorem*{uncon}{Conjecture}

\theoremstyle{definition}

\newtheorem{remark}[theorem]{Remark}

\newtheorem{obs}[theorem]{Observation}

\begingroup
    \makeatletter
    \@for\theoremstyle:=definition,remark,plain\do{%
        \expandafter\g@addto@macro\csname th@\theoremstyle\endcsname{%
            \addtolength\thm@preskip\parskip
            }%
        }
\endgroup

\newcommand{\df}[1]{{{\color{black!50!blue}\em #1}}}
\newcommand{\mb}[1]{\mathbb{#1}}

\numberwithin{equation}{section}

\DeclareMathOperator{\conv}{conv}

\renewcommand{\epsilon}{\varepsilon}
\renewcommand{\phi}{\varphi}

\newcommand{\n}{\newline}

\newcommand{\fig}{{\sc Figure }}
\newcommand{\cal}{\mathcal}

\newcommand{\upstart}{ .. controls +(0.5,0) and +(-0.2,-0.2) .. ++(1,0.5) }
\newcommand{\upend}{ .. controls +(0.2,0.2) and +(-0.5,0) .. ++(1,0.5)}
\newcommand{\downstart}{ .. controls +(0.5,0) and +(-0.2,0.2) .. ++(1,-0.5)}
\newcommand{\downend}{ .. controls +(0.2,-0.2) and +(-0.5,0) .. ++(1,-0.5)}
\newcommand{\downup}{ .. controls +(0.3,-0.3) and +(-0.3,-0.3) .. ++(1,0)}
\newcommand{\updown}{ .. controls +(0.3,0.3) and +(-0.3,0.3) .. ++(1,0)}

\newcommand{\ds}{\downstart}
\newcommand{\de}{\downend}
\newcommand{\us}{\upstart}
\newcommand{\ue}{\upend}
\newcommand{\ud}{\updown}
\newcommand{\du}{\downup}

\newcommand{\ft}{\footnotesize}
\newcommand{\lu}{--++(1,1)}
\newcommand{\ls}{--++(1,0)}
\newcommand{\ld}{--++(1,-1)}

\usetikzlibrary{fit}
\tikzset{%
  highlight/.style={rectangle,rounded corners,fill=red!15,draw,fill
    opacity=0.5,thick,inner sep=0pt} }

\usetikzlibrary{shapes.geometric,matrix,positioning,calc}

\begin{document}

\title{The  Erd\H{o}s-Szekeres problem for non-crossing convex sets}

\author{Michael G. Dobbins}
\address{M. Dobbins \\ GAIA \\ Postech \\ Pohang \\ South Korea}
\email{dobbins@postech.ac.kr}

\author{Andreas F. Holmsen}
\address{A. Holmsen \\ Department of Mathematical Sciences \\ KAIST \\
  Daejeon \\ South Korea} 
\email{andreash@kaist.edu}

\author{Alfredo Hubard}
\address{A. Hubard \\ D\'{e}partement d'informatique \\ \'{E}cole Normale
  Sup\'{e}rior \\Paris \\ France}
\email{hubard@di.ens.fr}

\maketitle

\begin{abstract}
We show an equivalence between a conjecture of Bisztriczky and Fejes
T{\'o}th about arrangements of planar convex bodies and a conjecture of
Goodman and Pollack about point sets in topological affine planes. As
a corollary of this equivalence we improve the upper bound of Pach and
T\'{o}th on the Erd\H{o}s-Szekeres theorem for disjoint convex bodies,
as well as the recent upper bound obtained by Fox, Pach, Sudakov and
Suk, on the Erd\H{o}s-Szekeres theorem for non-crossing convex
bodies. Our methods also imply improvements on the positive fraction
Erd\H{os}-Szekeres theorem for disjoint (and non-crossing) convex
bodies, as well as a generalization of the partitioned
Erd\H{o}s-Szekeres theorem of P\'{o}r and Valtr to arrangements of
non-crossing convex bodies. 
\end{abstract}

\section{Introduction}

\subsection{The happy ending theorem} In 1935, Erd\H{o}s and Szekeres
proved the following foundational result in combinatorial geometry and
Ramsey theory.\footnote{Paul Erd{\H o}s colloquially referred to this
  as the ``happy ending theorem'' as it led to the meeting of George
  Szekeres and Esther Klein, who went on to get married and live
  happily ever after $\dots$} 

\begin{unthm}[Erd\H{o}s-Szekeres \cite{erd-sze1}] \label{original} For
  every integer $n\geq 3$ there exists a minimal positive integer
  $f(n)$ such that any set of $f(n)$ points in the Euclidean plane, in
  which every triple is convexly independent, contains $n$ points
  which are convexly independent. 
\end{unthm}

Here {\em convexly independent} means that no point is contained in
the convex hull of the others. Determining the precise growth of the
function $f(n)$ is one of the longest-standing open problems of
combinatorial geometry, and has generated a considerable amount of
research. For history and details, see \cite{Barakaro, morris} and the
references therein. Two proofs are given in \cite{erd-sze1}, one of
which shows that $f(n)\leq \binom{2n-4}{2n-2}+1$, and in
\cite{erd-sze2} Erd\H{o}s and Szekeres give a construction showing
that $f(n) \geq 2^{n-2} + 1$.    

\begin{uncon}[Erd\H{o}s-Szekeres] \label{ES-conj}
  $f(n) = 2^{n-2}+1$.
\end{uncon}

This conjecture has been verified for $n\leq 6$ \cite{erd-sze1,
  szekeres17}, while for $n>6$ the best known upper bound is $f(n)
\leq \binom{2n-5}{n-2}+1 \sim 4^n/\sqrt{n}$, which is due to T{\'o}th
and Valtr \cite{totval}. Asymptotically this is the same as the bound
given by Erd\H{o}s and Szekeres in their seminal paper.  

\subsection{Generalized configurations} It was observed by Goodman and
Pollack \cite{goopolES} that the Erd\H{o}s-Szekeres theorem extends to
so-called {\em generalized configurations}, i.e. point sets in a
topological affine plane \cite{GPsemi, topaff, grunbaumS}. One may consider
this as a finite configuration of points in the plane where each pair
of points are contained in a unique {\em pseudoline} in such a way
that the resulting set of pseudolines form a {\em pseudoline
  arrangement} \cite{goody}. This underlying pseudoline arrangement
induces a convexity structure on the point configuration in a natural
way, as each pair of points span a unique {\em pseudosegment} which is
contained in their associated pseudoline. We define the convex hull of
a pair of points to be their connecting pseudosegement, from which it
follows that a triple of points is convexly independent if each pair
of points determine distinct pseudolines. The complement of the
pseudosegements determined by the configuration is a collection of
open regions, one of which is unbounded, and the {\em convex hull} of
the configuration is the complement of the unbounded region. See \fig
\ref{gencon}. It turns out that many basic theorems of convexity hold
in this more general setting, for instance, a set of points is
convexly independent if and only if every four of its points are
convexly independent \cite{smoro} (which is commonly called
Carath\'{e}odory's theorem). 

Generalized configurations also have a purely combinatorial
characterization and there are several equivalent axiom systems which
define them. Other names for generalized configurations which can be
found in the literature are {\em uniform rank 3 acyclic oriented
  matroids} \cite{OMS} or {\em CC-systems} \cite{knuti}.  

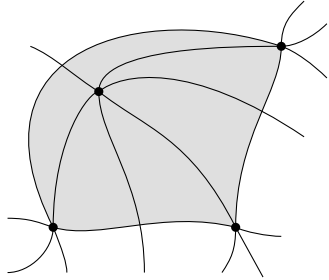
\begin{figure}[h!]
\centering
\begin{tikzpicture}
\begin{scope}[scale = .6]

\fill[gray!50, opacity=.5]
(0,0) .. controls (-2,4) and (2,5) .. (5,4) .. controls (5,3) and (4,2) .. (4,0).. controls (2.4,.4) and (1,-.3) .. (0,0) ;
\fill[black] 
(0,0) circle [radius = .1cm] (1,3) circle [radius = .1cm] (5,4) circle [radius = .1cm] (4,0) circle [radius = .1cm];
\draw (.3,-1) .. controls (.3,-.8) and (.2,-.5) .. (0,0) .. controls (-2,4) and (2,5) .. (5,4) .. controls (5.3,3.9) and (5.7,3.6) .. (6,3.3);    
\draw (-1,-1) .. controls (-.5,-1) and (0,-.5) .. (0,0) .. controls (0,2) and (.8,3) .. (1,3) .. controls (1.5,3.4) and (3,3.7) .. (5.5,2);
\draw (-1,.2) .. controls (-.5,.2) and (-.3,.1) .. (0,0) .. controls (1,-.3) and (2.4,.4) .. (4,0) .. controls (4.2, -.1) and (4.7,-.2) .. (5,-.2);
\draw (5.5,5) .. controls (5.2,4.7) and (5,4.5) .. (5,4) .. controls (5,3) and (4,2) .. (4,0) .. controls (4,-.5) and (3.9, -.7) .. (3.7,-1);
\draw (4.6, -1.1) -- (4,0) .. controls (3,2) and (2,2.2) .. (1,3) .. controls (.5, 3.3) and (0,3.8) .. (-.5, 4);
\draw (6,4.5) .. controls (5.7, 4.2) and (5.3, 4) .. (5,4) .. controls
(4,4) and (1,4) .. (1,3) .. controls (1,2) and (2,1) .. (2,-1); 

\end{scope}
\end{tikzpicture}
\caption{\ft Four point in a topological plane. Each pair of points
  determines a unique pseudoline. Their convex hull is the shaded
  region and shows that the points are not convexly independent.} 
\label{gencon}
\end{figure}

\begin{unthm}[Goodman-Pollack \cite{goopolES}] 
  For every integer $n\geq 3$ there exists a minimal positive integer
  $g(n)$ such that any generalized configuration of size $g(n)$,
  in which every triple is convexly independent,  contains $n$ points
  which are convexly independent. 
\end{unthm}

It should be noted that this is a proper generalization of the
Erd\H{os}-Szekeres theorem as there are substantially more
combinatorially distinct point sets in topological affine planes than
there are in the Euclidean plane \cite{fels-valt, gp-upper}. By
containment it follows that $f(n)\leq g(n)$.  

\begin{uncon}[Goodman-Pollack] $f(n) = g(n)$. \end{uncon}

It is not difficult to extend the proof of T\'{o}th and Valtr to
generalized configurations, as their proof uses no metric
properties. This appears to be a well-known fact, but since we could
not find any proof of this in the literature we include one in
section \ref{uppers}.
We therefore have $g(n)\leq \binom{2n-5}{n-2}+1$ for
$n\geq 7$. Also, the computer aided proof of Szekeres and Peters
\cite{szekeres17} confirming that $f(6)=17$, actually encodes
generalized configurations and it follows that $g(n) = f(n) =
2^{n-2}+1$ for all $n\leq 6$.     

\subsection{Mutually disjoint convex bodies}
In a different direction, initiated by Bisztriczky and Fejes T\'{o}th,
the Erd\H{o}s-Szekeres theorem was generalized to arrangements of
compact convex sets in the plane (which we call {\em bodies} for
brevity). An arrangement of bodies is {\em convexly independent} if no
member is contained in the convex hull of the others. 

\begin{unthm}[Bisztriczky-Fejes T\'{o}th \cite{biszFT1}] For any
  integer $n\geq 3$ there exists a minimal positive integer $h_0(n)$
  such that any arrangement of $h_0(n)$ pairwise disjoint bodies in
  the Euclidean plane, in which every triple is convexly independent,
  contains an $n$-tuple which is convexly independent.   
\end{unthm} 

This reduces to the Erd\H{o}s-Szekeres theorem when the bodies are
points, but was somewhat more complicated to establish in general. The
added complexity is reflected in the original upper bound  $h_0(n)
\leq t_n(t_{n-1}( \dots t_1(cn) \dots)) $, where $t_n$ is the $n$-th
tower function. The upper bound was later reduced to $16^n/n$  by Pach
and T\'{o}th in \cite{PachTothBodies}. By containment we have
$f(n)\leq h_0(n)$.  

\begin{uncon}[Bisztriczky-Fejes T\'{o}th]
  $f(n) = h_0(n)$.
\end{uncon}

Before this work the only known exact values are $h_0(4) = 5$ and
$h_0(5) = 9$ which were established in \cite{biszFT2}. 

\subsection{Non-crossing convex bodies}
The disjointness hypothesis was relaxed by Pach and T\'{o}th, who
showed that an Erd\H{o}s-Szekeres theorem also holds for arrangements
of {\em non-crossing} bodies, which means that for any pair of bodies
$A$ and  $B$, the set $A \setminus B$ is simply
connected. Equivalently, it means that $A$ and $B$ have precisely two
{\em common supporting tangents}. 

\begin{unthm}[Pach-T\'{o}th \cite{PachToth1}] For any integer $n\geq
  3$ there exists a minimal positive integer $h_1(n)$ such that any
  arrangement of $h_1(n)$ non-crossing bodies in the Euclidean plane,
  in which every triple is convexly independent, contains an $n$-tuple
  which is convexly independent.    
\end{unthm}

The original upper bound on $h_1(n)$ was improved to a doubly
exponential function in \cite{hubsuk}. Recently Fox, Pach , Sudakov,
and Suk \cite{FPSS} obtained the upper bound $h_1(n) \leq  2^{O(n^2
  \log n)}$. See also \cite{bt3, geza} for related work. 

The known bounds are summarized as follows.

\begin{figure}[h!]\centering
\begin{tikzpicture}
\node at (-.4,2) {$2^{n-2}+1$};
\node at (0.7,2) {$\leq$};
\node at (1.4,2) {$f(n)$};
\node at (2.1,2) {$\leq$};
\node at (2.8,2) {$g(n)$};
\node at (3.5,2) {$\leq$};
\node at (4.7,2) {$\binom{2n-5}{n-1}+1$};
\node at (7,2) {(for $n\geq 7$)};

\node at (1.4,1.2) {$f(n)$};
\node at (2.1,1.2) {$\leq$};
\node at (2.8,1.2) {$h_0(n)$};
\node at (3.5,1.2) {$\leq$};
\node at (4.45,1.2) {$\binom{2n-4}{n-2}^2$};

\node at (2.8,0.4) {$h_0(n)$};
\node at (3.5,0.4) {$\leq$};
\node at (4.3,0.4) {$h_1(n)$};
\node at (5.1,0.4) {$\leq$};
\node[right] at (5.3,0.47) {$2^{O(n^2 \log n)}$};  

\node[right] at (-1.25,-.6) {$2^{n-2}+1 = f(n) = g(n)$};
\node[right] at (4,-.6) {(for $n\leq 6$)};

\node[right] at (-1.25,-1.3) {$2^{n-2}+1 = h_0(n)$};
\node[right] at (4,-1.3) {(for $n\leq 5$)};

\end{tikzpicture}
\end{figure}

\subsection{Our results}
In this paper we make considerable improvements on $h_0(n)$ and
$h_1(n)$ by establishing the following.   

\begin{theorem} \label{equivES} The Erd\H{o}s-Szekeres problems for
  generalized configurations and for arrangements of non-crossing
  bodies are equivalent. In other words, $g(n) = h_1(n)$.  
\end{theorem}

Here is the idea of the proof. For the lower bound we use the fact
that a generalized configuration has a dual representation as a marked
pseudoline arrangement, i.e. a wiring diagram \cite{folkman,
  goodburr}. Using this representation we show that every generalized
configuration can be represented by an arrangement of bodies in the
Euclidean plane. This shows that $f(n)\leq h_1(n)$.  

To establish the reverse inequality we start with an arrangement of
bodies and consider its dual system of support curves drawn on the
cylinder $\mb{S}^1\times \mb{R}^1$. This system of curves induces a
cell complex which encodes the convexity properties of the
arrangement. We show how to modify this complex by elementary
operations, similar to those of Habert and Pocchiola \cite{habert} and
Ringel \cite{ringel}, while maintaining control of the convexity
properties of the arrangement. The process ends with a complex induced
by an arrangement representing a generalized configuration. The
details are given in section \ref{proof1}. 

In view of Theorem \ref{equivES} we obtain the following bounds.

\begin{figure}[h!]\centering
\begin{tikzpicture}
\node at (0,3) {$2^{n-2}+1 \leq f(n) \leq h_0(n) \leq h_1(n) = g(n)
  \leq \binom{2n-5}{n-2}+1$};
\node at (6,3) {(for $n\geq 7$)};
\node at (0,2.2) {$2^{n-2}+1 = f(n) = h_0(n) = h_1(n) = g(n)$};
\node at (6,2.2) {(for $n\leq 6$)};
\end{tikzpicture}
\end{figure}

Our proof actually provides a general procedure for reducing
non-crossing arrangements to generalized configurations, and can
therefore be applied to the multitude of Erd\H{o}s-Szekeres-type
results previously proven separately for point sets, then for
arrangements of bodies. (See for instance \cite{baranyES, pach-soly1,
  partitionES, por-convex}.)  In particular we obtain the positive
fraction version and the partitioned version of the Erd\H{o}s-Szekeres
theorem for non-crossing arrangements. This will be discussed in
section \ref{extensions}. 

\begin{remark}
It should be noted that the condition that the bodies are convex is
not strictly necessary. The theorems also hold for any arrangement of
compact sets $\{A_1, \dots, A_n\}$ provided we impose conditions
on the arrangement $\{\conv (A_1), \dots, \conv(A_n)\}$. However, no
real generality is gained by this formulation, so we restrict
ourselves to arrangements of convex bodies to make our statements simpler. 
\end{remark}

\section{Proof of Theorem \ref{equivES}} \label{proof1}

\subsection{Duality} We call a compact convex subset of $\mb{R}^2$ a
\df{body}, and a finite collection of at least three bodies an
\df{arrangement}. For a body $A$, recall its support function $h_A :
\mb{S}^1 \to \mb{R}^1$ on the unit circle \[h_A(\theta) := \max_{x\in
  A} \langle \theta, x \rangle\] The \df{dual} of a body $A$ is the graph
of its support function drawn on the cylinder
$\mb{S}^1\times\mb{R}^1$, i.e. \[A^* := \{(\theta, h_A(\theta)):
\theta \in \mb{S}^1\}\] We implicitly assume that the unit circle is
oriented, and the dual curves should therefore be thought of as {\em
  directed} curves. It is important to make a consistent choice of
orientation throughout, so  we fix
the positive orientation to be the {\em counter-clockwise direction}. 

Notice that for each $\theta \in \mb{S}^1$ the value $h_A(\theta)$
measures the oriented distance from the origin to the directed
\df{supporting tangent} of $A$ with direction $\theta + \frac{\pi}{2}$
which has the body $A$ on its left side. For instance, the dual of a
point distinct from the origin is a sine curve, while the dual of a
disk centered at the origin is the graph of the constant
function. Note that if $h_A(\theta)\geq 0$ for all $\theta$, then the
origin is contained in the body $A$. See \fig \ref{hull} (left). 

We use the term \df{system} when referring to a finite collection of
at least three curves on $\mb{S}^1 \times \mb{R}^1$ which are graphs
of continuous functions $\gamma: \mb{S}^1 \to \mb{R}^1$. In this way
every arrangement $\cal A$ is associated with its dual system $\cal
A^*$. Notice that a body is uniquely determined by its support
function (see section 2.2 in \cite{groemer}), and consequently an
arrangement is uniquely determined by its dual system. Moreover, if a
pair of dual curves intersect, that is, if $h_A(\theta) = h_B(\theta)$
for some $\theta\in \mb{S}^1$, then bodies the $A$ and $B$ have a
\df{common supporting tangent} in the direction $\theta
+\frac{\pi}{2}$. The \df{upper envelope} of a system given by
functions $\{\gamma_i\}$ is the graph of the function $\max_i
\gamma_i$. The following observation implies that the convexity
properties of an arrangement can be determined by its dual system.  

\begin{obs} \label{hullenvy}
An arrangement is convexly independent if and only if every curve appears on the upper envelope of the dual system.
\end{obs}

To see why this holds, notice that an arrangement $\cal A$ is convexly
independent if and only if for any body $B\in {\cal A}$, the convex
hull of the union of the members of $\cal A$ can be supported by a
supporting tangent of $B$ which is disjoint from every member of $\cal
A\setminus \{B\}$. If $\alpha$ is the direction of such a supporting
tangent of $B$, then $h_B(\alpha-\frac{\pi}{2}) >
h_A(\alpha-\frac{\pi}{2})$ for every $A \in {\cal A}\setminus
\{B\}$. See \fig \ref{hull} (right). 

\begin{figure}[h!]
  \centering
  
\begin{tikzpicture}

\begin{scope}[scale=.7, xshift = -11cm]
    
\begin{scope}[rotate = 47]
\fill[lightgray, opacity = .5]
(3,0.5) .. controls +(0:1) and +(70:1) .. 
(5,-1.5) .. controls +(70:-1) and +(-5:1) ..
(2,-3.7) .. controls +(-5:-1) and +(80:-1) .. 
(1,-1.5) .. controls +(80:1) and +(0:-1) ..
(3,0.5);
\draw[thick]
(3,0.5) .. controls +(0:1) and +(70:1) .. 
(5,-1.5) .. controls +(70:-1) and +(-5:1) ..
(2,-3.7) .. controls +(-5:-1) and +(80:-1) .. 
(1,-1.5) .. controls +(80:1) and +(0:-1) ..
(3,0.5);
\end{scope}

\begin{scope}[xshift =3cm, yshift= 1cm]

\begin{scope}
\draw[-latex](0,0) -- (3.4,0);
\draw[-latex](0,0) ++(0:1.2cm) arc (0:37:1.2cm);   
\node at (.9,.33) {\ft $\theta$};
\end{scope}

\begin{scope}[rotate =37]
\draw[-latex](0,0) -- (3.4,0);  
\end{scope}

\begin{scope}[rotate=37, xshift = 2.21cm]
\draw[black!50!blue,-latex](0,-3) -- (0,3);
\draw[black!50!blue](0,.2) -- (.2,.2) -- (.2,0);
\end{scope}

\end{scope}

\end{scope}

\begin{scope}[scale = .5, rotate = 18]

\begin{scope}[xshift = 1cm, yshift = .5cm]
 \fill[lightgray, opacity = .5]
(0,0) .. controls +(0:1) and +(50:1) .. 
 (2,-2) .. controls +(50:-1) and +(-5:1) ..
 (-1,-3) .. controls +(-5:-1) and +(0:-1) ..
 (0,0);
 \draw[thick]
 (0,0) .. controls +(0:1) and +(50:1) .. 
 (2,-2) .. controls +(50:-1) and +(-5:1) ..
 (-1,-3) .. controls +(-5:-1) and +(0:-1) ..
 (0,0);
\end{scope}

\begin{scope}[yshift = 1.6cm, xshift = -.3cm]
\fill[lightgray, opacity = .5]
(3,0.5) .. controls +(0:1) and +(70:1) .. 
 (5,-1.5) .. controls +(70:-1) and +(-5:1) ..
 (2,-2.7) .. controls +(-5:-1) and +(80:-1) .. 
 (1,-1.5) .. controls +(80:1) and +(0:-1) ..
 (3,0.5);
\draw[thick]
 (3,0.5) .. controls +(0:1) and +(70:1) .. 
 (5,-1.5) .. controls +(70:-1) and +(-5:1) ..
 (2,-2.7) .. controls +(-5:-1) and +(80:-1) .. 
 (1,-1.5) .. controls +(80:1) and +(0:-1) ..
 (3,0.5);
 \end{scope}

\begin{scope}[yshift = -1cm, rotate = 45]
\path (-1,-3) ++(-5:5) coordinate (x);
\fill[lightgray, opacity = .5]
(5,0) .. controls +(0:2) and +(0:1) ..
 (x) .. controls +(0:-2) and +(0:-1) ..
 (5,0);
\draw[thick]
 (5,0) .. controls +(0:2) and +(0:1) ..
 (x) .. controls +(0:-2) and +(0:-1) ..
 (5,0);
\end{scope}

\begin{scope}[scale = -.9, xshift = -4cm]
\fill[lightgray, opacity = .5]
(3,0.5) .. controls +(0:1) and +(70:1) .. 
 (5,-1.5) .. controls +(70:-1) and +(-5:1) ..
 (2,-3.7) .. controls +(-5:-1) and +(80:-1) .. 
 (1,-1.5) .. controls +(80:1) and +(0:-1) ..
 (3,0.5);
 \draw[thick]
 (3,0.5) .. controls +(0:1) and +(70:1) .. 
 (5,-1.5) .. controls +(70:-1) and +(-5:1) ..
 (2,-3.7) .. controls +(-5:-1) and +(80:-1) .. 
 (1,-1.5) .. controls +(80:1) and +(0:-1) ..
 (3,0.5);
\end{scope}

\begin{scope}[yshift = -2cm, xshift =5.5cm ]
\draw[blue!50!black,-latex] (0,-1) -- (0,6);
\draw (0,2.8) -- (.3,2.8) -- (.3,2.5);
\draw[-latex] (-7,2.5) -- (2,2.5);

\node [above] at (0,6) {\ft $\alpha$};
\node [right] at (2,2.5) {\ft $\alpha-\frac{\pi}{2}$};
\end{scope}

\end{scope}
\end{tikzpicture}
\caption{\ft {\bf Left:} The support function $h_A(\theta)$ measures the distance between the origin and the directed supporting tangent in the direction $\theta+\frac{\pi}{2}$. {\bf Right:} If a body appears on the convex hull, then it has a supporting tangent which also supports the convex hull of the union of the members of the arrangement.}
\label{hull}
\end{figure}
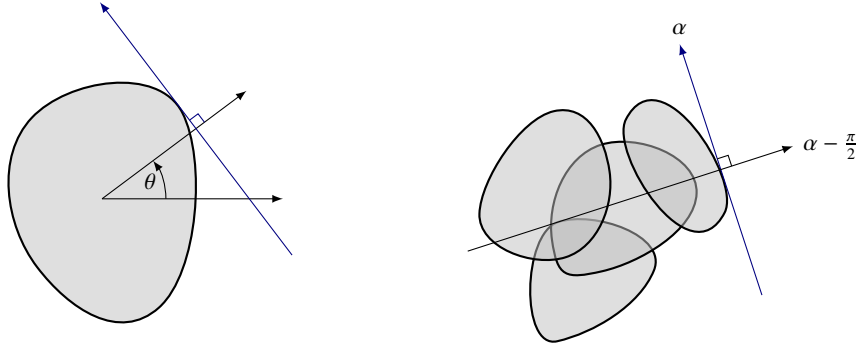

\subsection{Generic arrangements}

An arrangement is \df{generic} if the following hold. 

\begin{itemize}
\item For any pair of bodies $A_1$ and $A_2$ with common supporting 
  tangent $\ell$ the intersection $A_1 \cap A_2 \cap \ell$ is empty. 
\item No triple of bodies share a common supporting tangent. 
\end{itemize}

A standard perturbation argument shows that the optimal values for
$h_1(n)$ can be attained by generic arrangements, so hereby {\em all
  arrangements are assumed to be generic}. The conditions above imply
the following for the dual system. 

\begin{obs} \label{generic} For the dual system of a generic
  arrangement, each pair of curves intersect transversally and no
  three curves intersect in a common point. \end{obs}

\subsection{Review of generalized configurations} \label{generalizedC}

In the sequel it will be useful to recall the duality between
generalized configurations and pseudoline arrangements. This is
a combinatorial version of the classical projective duality between
points and lines that extends to the realm of generalized
configurations and pseudoline arrangements. From a combinatorial point
of view this is very similar to our duality for arrangements of
bodies, and the connection is crucial for relating generalized
configurations to arrangements of bodies.

We start by recalling the notion of the {\em allowable sequence} of a
set of points in the plane. We will assume that this point set is in
a strongly general position, meaning no three points are collinear and
no two lines determined by the points are parallel.\footnote{The
general theory developed by Goodman and Pollack \cite{goodburr,
  GPallow} does not require this assumption, but for us it is no loss
of generality.} Let $P$ be a set of $n$
labeled points in strongly general position in the plane and consider a generic
directed line 
$l_1$. If we project the points orthogonally onto the line $l_1$, then
the direction of $l_1$ will induce a linear ordering of the points which
we record as a permutation $\pi_1 = \pi_1(P)$. As the line rotates
counter-clockwise about a fixed point this ordering will change each
time the line becomes orthogonal to a direction determined by a pair
of points in $P$, resulting in a periodic sequence of
permutations \[\dots, \pi_1, \pi_2, \dots, \pi_{n(n-1)}, \pi_1, \pi_2,
\dots\] which is called the allowable sequence of $P$. Notice that the
allowable sequence satisfies the following properties
\begin{enumerate}
\item Any two consecutive terms $\pi_i$ and $\pi_{i+1}$ differ by
  reversing the order of two adjacent elements.
\item In any $\binom{n}{2}$ consecutive terms of the sequence each
  pair of elements of $P$ switches exactly once.
\end{enumerate}
It is an immediate consequence that for all $i$, the permutation
$\pi_{i+\binom{n}{2}}$ is the reverse of $\pi_i$. Every allowable
sequence determines a periodic sequence of ordered switches. That is,
rather than writing down each permutation, we only record which order
pair switches at each step. The convention is to record the order of
the pair {\em before} they switch, so for instance, the consecutive
pair of permutations  
\[(\cdots, a, b, \cdots) \rightarrow (\cdots, b, a,\cdots)\]
will be recorded as the ordered switch $ab$, and consequently,
$\binom{n}{2}$ steps later  
we get the ordered switch $ba$. It turns out that an allowable
sequence is determined by its sequence of ordered switches, which is shown
in Proposition 2.6 of \cite{GPallow}. For instance, the following half-period
of a sequence of ordered switches
\[\dots, dc, ac, bc, ad, bd, ba, \dots\]
uniquely determines the following sequence of permutations
\[\dots, ba\underline{dc}, b\underline{ac}d, \underline{bc}ad,
cb\underline{ad}, c\underline{bd}a, cd\underline{ba}, cdab, \dots\]

More generally, any sequence of permutations which satisfies properties
(1) and (2) is called an \df{allowable sequence}, but not every such
sequence can be obtained from a set of points in the plane by the
procedure described above. This is where generalized configurations
come in to play: For an ordered pair of points, $a$ and $b$, of a
generalized configuration, consider  the directed pseudoline which
first passes through $a$ then through $b$, and label the point where
it intersects the distinguished 
line at infinity by the ordered pair $ab$ (thus the antipodal point is
labeled $ba$). In this way we obtain a cyclic sequence of the ordered
pairs of points which is antipodal in the sense that a half-period
after the term $ab$ we get the reverse pair $ba$. It turns out that
this sequence of ordered pairs is precisely the sequence of ordered
switches of an allowable sequence, and Theorem 4.4 of \cite{GPsemi}
shows that {\em every allowable sequence} can be obtained in this
way. 

Clearly an allowable sequence is uniquely determined by any one of its
half-periods, which can be encoded by a so-called
\df{wiring diagram}, resulting in the combinatorial dual pseudoline
arrangement. Let 
$\pi_1, \dots, \pi_N$ be the permutations of  
some half period of the allowable sequence and $S_i$ the ordered
switch from $\pi_i$ to $\pi_{i+1}$ (where $\pi_{N+1}$ is the reverse
of $\pi_1$). 
 Construct the wiring diagram, starting with horizontal
``wires'' going from left to right, and labeled by the elements of the
permutations in the order in which they appear in the 
permutation $\pi_1$ from bottom to top. Apply the switch $S_1$ by
crossing the wires corresponding to the elements appearing in the
switch $S_1$. After all switches have been applied, each pair of wires
will have crossed precisely once, and we arrive at the reverse of the
initial permutation. See \fig \ref{allowable}.

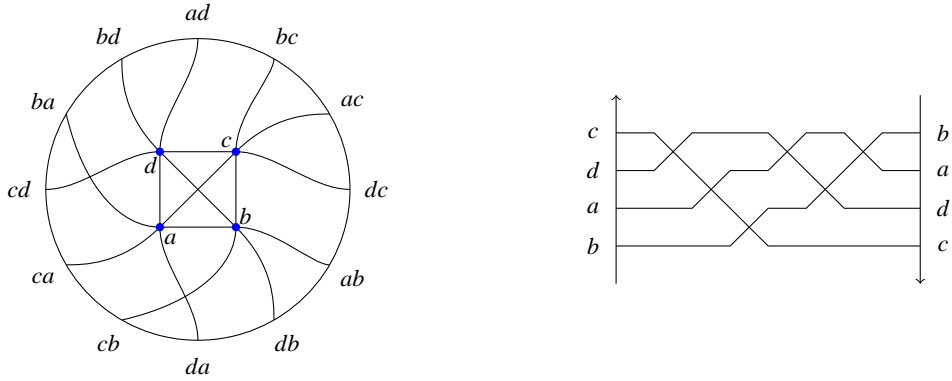
\begin{figure}[h!]
\centering

\begin{tikzpicture}
  
\begin{scope}[scale = .5]
\draw (0,0) circle [radius = 4cm];
\draw (180:4cm) .. controls (-3,0) and (-2,1) .. (-1,1) -- (1,1)
.. controls (2,1) and (3,0) .. (0:4cm); 
\draw (150:4cm) .. controls (-3,0) and (-2,-1) .. (-1,-1) -- (1,-1)
.. controls (2,-1) and (3,-2) .. (330:4cm); 
\draw (210:4cm) .. controls (-3,-2) and (-2,-2) .. (-1,-1) -- (1,1)
.. controls (2,2) and (3,2) .. (30:4cm); 
\draw (240:4cm) .. controls (0,-3) and (1,-2) .. (1,-1) -- (1,1)
.. controls (1,2) and (2,3) .. (60:4cm); 
\draw (270:4cm) .. controls (0,-3) and (-1,-2) .. (-1,-1) -- (-1,1)
.. controls (-1,2) and (0,3) .. (90:4cm); 
\draw (300:4cm) .. controls (2,-3) and (2,-2) .. (1,-1) -- (-1,1)
.. controls (-2,2) and (-2,3) .. (120:4cm); 
\fill[blue] (1,1) circle [radius = .11cm];
\fill[blue] (-1,1) circle [radius = .11cm];
\fill[blue] (-1,-1) circle [radius = .11cm];
\fill[blue] (1,-1) circle [radius = .11cm];
\node at (-.72,-1.3) {\ft $a$};
\node at (1.25, -.7) {\ft $b$};
\node at (.75,1.25) {\ft $c$};
\node at (-1.24,.65) {\ft $d$};
\node at (0:4.7cm) {\ft $dc$};
\node at (30:4.7cm) {\ft $ac$};
\node at (60:4.7cm) {\ft $bc$};
\node at (90:4.7cm) {\ft $ad$};
\node at (120:4.7cm) {\ft $bd$};
\node at (150:4.7cm) {\ft $ba$};
\node at (180:4.7cm) {\ft $cd$};
\node at (210:4.7cm) {\ft $ca$};
\node at (240:4.7cm) {\ft $cb$};
\node at (270:4.7cm) {\ft $da$};
\node at (300:4.7cm) {\ft $db$};
\node at (330:4.7cm) {\ft $ab$};
\end{scope}


\begin{scope}[yshift = -.75cm, xshift = 5.5cm, scale = .5]

\node at (-.6,3) {\ft $c$};
\node at (-.6,2) {\ft $d$};
\node at (-.6,1) {\ft $a$};
\node at (-.6,0) {\ft $b$};

\draw(0,3) \ls\ld \ld \ld \ls \ls \ls \ls;
\draw(0,2) \ls\lu \ls \ls \ld \ld \ls \ls;
\draw(0,1) \ls\ls \lu \ls \lu \ls \ld \ls;
\draw(0,0) \ls\ls \ls \lu \ls \lu \lu \ls;

\draw[->] (0,-1) -- (0,4); 
\draw[<-] (8,-1) -- (8,4); 

\node at (8.6,0) {\ft $c$};
\node at (8.6,1) {\ft $d$};
\node at (8.6,2) {\ft $a$};
\node at (8.6,3) {\ft $b$};
\end{scope}

\end{tikzpicture}

\caption{\ft {\bf Left:} A generalized configuration with the labeling on the line
  at infinity. {\bf Right:}The dual wiring diagram  corresponding to the
  half-period of ordered switches $(dc,ac,bc,ad,bd,ba)$.} 
\label{allowable}
\end{figure}

\subsection{Non-crossing and orientable
  arrangements} \label{orientable} 

A pair of bodies $A_1$, $A_2$ is \df{non-crossing} if $A_1 \setminus
A_2$ is connected, or equivalently, if $A_1$ and $A_2$ have precisely
two common supporting tangents. Notice that by Observation
\ref{generic}, the dual curves of a non-crossing pair of bodies will
meet in precisely two crossing points. 

A triple of bodies $A_1$, $A_2$, $A_3$ is \df{orientable} if every
pair is non-crossing and $\conv (A_i\cup A_j) \setminus A_k$ is simply
connected for all choices of distinct $i$, $j$, $k$, or equivalently,
the convex hull of $A_1 \cup A_2 \cup A_3$ is supported by exactly
three of the common supporting tangents determined by the pairs $A_i,
A_j$. A \df{non-crossing arrangement} is one in which each pair is
non-crossing, and an \df{orientable arrangement} is one in which each
triple is orientable.   

Each member of an orientable triple contributes a single connected arc
to the boundary of its convex hull, so traversing the boundary of its
convex hull in the counter-clockwise direction will impose a cyclic
ordering of the triple. Notice that in the dual system of an oriented
triple, each curve appears precisely once on the upper envelope in the
same cyclic order as the one we get by traversing the convex hull of
the bodies. See \fig \ref{orry}.

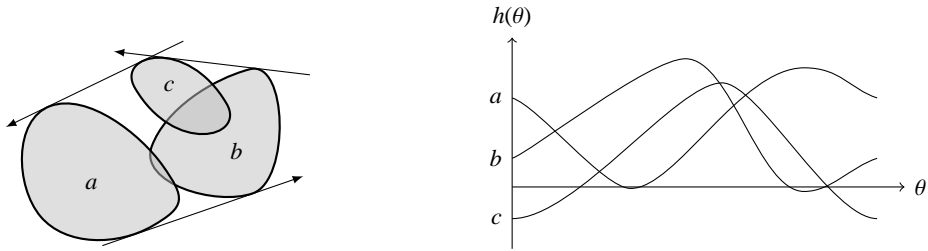
\begin{figure}[h!]
\centering

\begin{tikzpicture}
  
\begin{scope}[scale = .52]

\begin{scope}[rotate = 130]

\begin{scope}[yshift = -4cm, rotate = 70]
\fill[lightgray, opacity=.5]
(0,0) .. controls +(0:1) and +(50:1) .. 
 (2,-2) .. controls +(50:-1) and +(-5:1) ..
 (-1,-3) .. controls +(-5:-1) and +(0:-1) ..
 (0,0);
\draw[thick]
 (0,0) .. controls +(0:1) and +(50:1) .. 
 (2,-2) .. controls +(50:-1) and +(-5:1) ..
 (-1,-3) .. controls +(-5:-1) and +(0:-1) ..
 (0,0); 
\end{scope}

\begin{scope}
\fill[lightgray, opacity=.5]
(3,0.5) .. controls +(0:1) and +(70:1) .. 
 (5,-1.5) .. controls +(70:-1) and +(-5:1) ..
 (2,-2.7) .. controls +(-5:-1) and +(80:-1) .. 
 (1,-1.5) .. controls +(80:1) and +(0:-1) ..
 (3,0.5);
\draw[thick]
 (3,0.5) .. controls +(0:1) and +(70:1) .. 
 (5,-1.5) .. controls +(70:-1) and +(-5:1) ..
 (2,-2.7) .. controls +(-5:-1) and +(80:-1) .. 
 (1,-1.5) .. controls +(80:1) and +(0:-1) ..
 (3,0.5);
\path (-1,-3) ++(-5:5) coordinate (x);
\end{scope}

\begin{scope}[rotate = -10, yshift = -4.5cm, xshift = -2cm]
\fill[lightgray, opacity=.5]
(5,0) .. controls +(0:2) and +(0:1) ..
(x) .. controls +(0:-2) and +(0:-1) ..
(5,0);
\draw[thick]
 (5,0) .. controls +(0:2) and +(0:1) ..
 (x) .. controls +(0:-2) and +(0:-1) ..
 (5,0);
\end{scope}
\end{scope}

\begin{scope}[yshift = 1.2cm, rotate = -2.6, xshift = -.75cm]
  \draw[-latex] (0,0) --(5,2);
\end{scope}

\begin{scope}[xshift = 4.5cm, yshift = 5.59cm]
\draw[-latex] (0,0) -- (173.15: 5cm); 
\end{scope}

\begin{scope}[xshift = 1.3cm, yshift = 6.45cm]
\draw[-latex] (0,0) -- (206: 5cm);
\end{scope}

\begin{scope}
  \node at (110:3cm) {\ft $a$};
  \node at (54: 4.5cm) {\ft $b$};
  \node at (80: 5.4cm) {\ft $c$};
\end{scope}
\end{scope}

\begin{scope}[yshift = 1.4cm, xshift = 5cm, xscale = 3, scale = .4]
\draw 
(0,3) .. controls (.3,2.7) and (1,0) .. 
(1.3,0) .. controls (1.8,0) and (2.5,4) .. 
(3.2,4) .. controls (3.6,4) and (3.7,3.3) .. (4,3);

\draw 
(0,1) .. controls (.2,1.2) and (1.5,4.3) .. 
(1.9,4.3) .. controls (2.4,4.3) and (2.7,-.1) .. 
(3.2,-.1) .. controls (3.5,-.1) and (3.7,0.7) .. (4,1);

\draw 
(0,-1) .. controls (.7,-1) and (1.8,3.5) .. 
(2.3,3.5) .. controls (2.7,3.5) and (3.5,-1) .. (4,-1);

\draw[->] (0,-2) -- (0,5);
\draw[->] (0,.05) -- (4.3,.05);

\node [left] at (0,3) {\ft $a$};
\node [left] at (0,1) {\ft $b$};
\node [left] at (0,-1) {\ft $c$};

\node[right] at (4.3,.05) {\ft $\theta$};
\node[above] at (0,5) {\ft $h(\theta)$};
\end{scope}
\end{tikzpicture}
  
\caption{\ft {\bf Left:} A typical orientable triple with the three
  common supporting tangents which support the convex hull. Between
  consecutive common supporting tangents the boundary of the convex
  hull consists of a boundary arc of one of the bodies, which induces
  a cyclic ordering of the bodies. {\bf Right:} The dual system of an
  orientable triple. The cyclic order in which the curves appear on
  the upper envelope (when traversed from left to right) coincides
  with the cyclic order in which we meet the bodies when traversing
  the boundary of the convex hull in the counter-clockwise order.} 
\label{orry}
\end{figure}

It is easily verified that the set of cyclic orderings of all triples
of an orientable arrangement satisfy the chirotope axioms of a rank 3
uniform oriented matroid (Definition 3.5.3 of \cite{OMS}), or
equivalently the axioms of a 
CC-system (see Section 1 of \cite{knuti}).\footnote{Gr\"{u}nbaum
  implicitly makes this 
  observation in his discussion on planar arrangements of simple
  curves in Section 3.3 of \cite{grunbaumS}.}
This means that for every orientable arrangement ${\cal A}$, there
exists a generalized configuration $\cal P$ and a bijection $\phi:
\cal A \to \cal P$ which preserves the cyclic ordering of every
triple. (The cyclic ordering of a triple in a generalized
configuration is defined, as before, by traversing the boundary of its
convex hull in the counter-clockwise direction and reading off the
cyclic order in which we meet the points.) See \fig \ref{pappus}. 

\begin{figure}[h!] \centering  
\begin{tikzpicture}
\begin{scope}[scale= -.8]
\clip (0,0) circle (3.2);
\foreach \x in {0,72,...,288}
{\draw
 ({\x-72}:1) .. controls +({\x-18}:1) and +(0,0) .. ({\x-15}:3.2)
 ({\x-72}:1) -- + +({\x-18}:-4)
 (\x:1) .. controls +({\x-18}:1) and +(0,0) .. ({\x-27}:3.2)
 (\x:1) -- +({\x-18}:-5);}
\foreach \x in {0,72,...,288}
{\fill[blue]
  (\x:1) circle (2pt);
\fill[blue]
  ({\x-21}:3) circle (2pt);}
\end{scope}

\begin{scope}[scale=-0.4, xshift = 20cm]
\foreach \x in {0,...,4}
{\draw[fill=blue!25]
 ({72*\x+27}:0.9) +({72*\x}:-0.4) coordinate (c\x) -- +({72*\x}:0.4)
 coordinate (b\x) -- +({72*\x+90}:-1.2) coordinate (a\x) -- cycle; 
\draw[fill=blue!25]
 ({72*\x-17.5}:7) +({72*\x}:-0.4) -- +({72*\x}:0.4) --
 +({72*\x+90}:-1.2) -- cycle; }
\foreach \i in {0,...,4} 
{\draw[thin, dotted] 
 let \n1={int(mod(\i+2,5))} 
 in (a\n1) -- ($(a\n1)!6!(b\i)$);
\draw[thin, dotted]
 let \n1={int(mod(\i+1,5))} 
 in (a\i) -- ($(a\i)!6!(c\n1)$);
\draw[thin, dotted]
 let \n2={int(mod(\i+1,5))}
 in (b\i) -- ($(b\i)!6!(b\n2)$);}
\end{scope}
\end{tikzpicture}
\caption{\ft An arrangement of convex bodies (left) and a realization
  of this arrangement by a generalized configuration (right). In fact
  this configuration is based on Goodman and Pollack's ``bad
  pentagon'' \cite{GPallow} and can not be realized by points and
  straight lines.} \label{pappus} 
\end{figure}
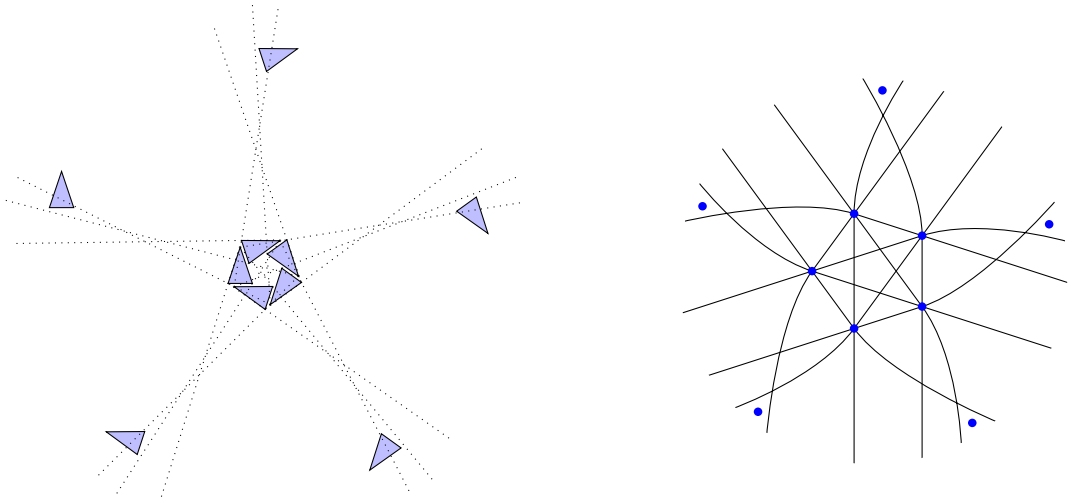

When there exists such an order-preserving bijection as described
above, we say that the arrangement is \df{realizable} by the
generalized configuration, and we may also say that the
generalized configuration is realizable by the arrangement. Our
discussion above implies the following.

\begin{lemma}\label{converserep} Every orientable arrangement is
  realizable by a generalized configuration. \end{lemma} 

We now establish the converse of Lemma \ref{converserep}.

\begin{lemma}\label{realz} Every generalized configuration is
  realizable by an orientable arrangement.  \end{lemma}

What is of importance to us is that the convex independencies can be
determined only from the set of cyclic orderings of the triples
\cite{OMS, goody, knuti}. Therefore Lemmas \ref{converserep} and
\ref{realz} imply that the Erd\H{o}s-Szekeres problems for generalized
configurations and for orientable arrangements are equivalent. This,
however, is not enough to prove Theorem \ref{equivES} since there
exist arrangements which are not orientable. Non-orientable
arrangements will be dealt with in the next section.

\begin{remark}
In view of Observations \ref{hullenvy} and \ref{generic} it should be
clear that for our purpose, the precise geometric information of the
dual system is not of major importance, but rather just the
intersection patterns of the dual curves which allow us to determine
the upper envelopes of any sub-system. In the figures below the dual
systems will be represented by ``schematic diagrams'' similar to
wiring diagrams. 
\end{remark}

\begin{proof}[Proof of Lemma \ref{realz}]
Let $\cal P$ be a generalized configuration. Consider its dual wiring
diagram $\cal W$ which encodes some half-period of the allowable
sequence of $\cal P$, as described in section \ref{generalizedC}. We
can view $\cal W$ as a system of curves drawn on the M\"{o}bius
strip, each pair crossing once with all crossing points
distinct. Extending $\cal W$ to its double cover, we obtain a system
of curves $\cal F$ on the cylinder $\mb{S}^1\times\mb{R}^1$, each pair
crossing twice, with all crossing points distinct.

We now notice that the cyclic ordering of a triple of $\cal P$ is
encoded in $\cal F$ by the order in which the corresponding triple of
curves appear on the upper envelope (of the triple) in the double
cover. To see why this happens it suffices to consider a triple
of points, so consider three points with cyclic order is
$(a,b,c)$. A full period of the corresponding sequence of ordered switches will be 
\[ab, ac, bc, ba, ca, cb\]
giving us the allowable sequence
\[\dots, abc, bac, bca, cba, cab, acb, abc, \dots\]
The wires that appear on the upper envelope of a full period of the
double cover of the corresponding wiring diagram are the last entries
of each permutation. See \fig \ref{triporder}.

\begin{figure}[h!]
  \centering
  
  \begin{tikzpicture}
    \begin{scope}[scale = .6]
      \begin{scope}[scale = .56] 
\draw (-1,3) -- (-1,-3);
\draw (-3, -2) -- (3,1);
\draw (-3,2) -- (3,-1);
\node at (3.6,1.2) {\ft $ab$};
 \node at (-3.6,-2.2) {\ft $ba$};
 
\node at (-3.6,2.2) {\ft $bc$};
 \node at (3.6,-1.2) {\ft $cb$};
 
\node at (-1,3.4) {\ft $ac$};
\node at (-1,-3.4) {\ft $ca$};

\node[left] at (-.8,-.7) {\ft $a$};
\node[below] at (1,0.1){\ft $b$};
\node[right] at (-1.2,1.2) {\ft $c$};
       \fill[blue] 
        (1,0) circle [radius = .12]
        (-1,1) circle [radius = .12] 
        (-1,-1) circle [radius = .12];
      \end{scope}
      \begin{scope}[xshift = 5.6cm, yshift = -.5cm, scale = .7] \draw
        (0,2) \ls\ls\ld\ld\ls\lu\lu\ls
        (0,1) \ls\ld\ls\lu\lu\ls\ld\ls
        (0,0) \ls\lu\lu\ls\ld\ld\ls\ls;
\node[left] at (0,2) {\ft $c$};
\node[left] at (0,1) {\ft $b$};
\node[left] at (0,0) {\ft $a$};
      \end{scope}
    \end{scope}
  \end{tikzpicture}

  \caption{\ft The cyclic ordering of each triple of points (right)
    corresponds to the order in which the wires appear on the upper
    envelope of the double cover of the wiring diagram (left).}
  \label{triporder}
\end{figure}
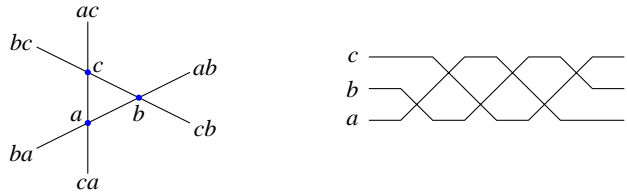

The curves of $\cal F$ can be approximated by the graph of a smooth
function $f : \mb{S}^1 \to \mb{R}^1$, and Blaschke showed that if we
interpret a smooth $2\pi$-periodic real function $f$ as the support
function of a planar curve in the plane, then the curvature at $t$ is
given by the expression $f(t)+f''(t)$ (see, for instance, Lemma 2.2.3 in
\cite{groemer}). Hence there exists a constant $c_0$ such
that $f + c$ is the support function of a body for any $c >
c_0$, and for a given system of smooth curves there is a common
constant we can add to each of the curves making $\cal F$ the dual
system of an orientable arrangement. The intersection patterns stay
invariant under this procedure and the result follows.
See \fig \ref{cover}.  
\end{proof}

\begin{figure}[h!]  
\centering
\begin{tikzpicture}[scale=.35]
\begin{scope}[xshift = -8cm]
\draw[dotted] 
(0,0) -- (1,-2) -- (2,3) -- (4,0) -- (7,-1) -- (0,0)
(0,0) -- (2,3) -- (7,-1) -- (1,-2) -- (4,0) -- (0,0);
\fill 
(0,0) circle [radius = .13cm];
\fill(1,-2) circle [radius = .13cm];
\fill(2,3) circle [radius = .13cm];
\fill(4,0) circle [radius = .13cm];
\fill[black](7,-1) circle [radius = .13cm];
\node [left] at (0,0) {\ft $a$};
\node [below] at (1,-2) {\ft $b$};
\node [above] at (2,3) {\ft $c$};
\node [above] at (4.2,0) {\ft $d$};
\node [below] at (7,-1) {\ft $e$};
\end{scope}

\begin{scope}[xshift = 10cm]
\draw
(0,2.5) \ld\ls\ls\ls\ld\ld\ls\ls\ld\ls\ls
(0,1.5) \lu\ls\ls\ls\ls\ls\ld\ld\ls\ld\ls
(0,.5)  \ls\ld\ld\ls\ls\ls\ls\ls\lu\lu\ls
(0,-.5) \ls\lu\ls\ld\ls\lu\ls\lu\ls\ls\ls
(0,-1.5)\ls\ls\lu\lu\lu\ls\lu\ls\ls\ls\ls;
\draw[<-](0,3.5) -- (0,-2.5);
\draw[<-](11,-2.5) -- (11,3.5);
\node[left] at (0,2.5) {\ft $a$};
\node[left] at (0,1.5) {\ft $b$};
\node[left] at (0,0.5) {\ft $c$};
\node[left] at (0,-.5) {\ft $d$};
\node[left] at (0,-1.5) {\ft $e$};
\node[right] at (11,2.5) {\ft $e$};
\node[right] at (11,1.5) {\ft $d$};
\node[right] at (11,0.5) {\ft $c$};
\node[right] at (11,-.5) {\ft $b$};
\node[right] at (11,-1.5) {\ft $a$};
\end{scope}

\begin{scope}[xshift=-4cm, yshift = -9cm]
\draw (0,2.5) 
\ds\de\ls\ls\ds\ld\de\ls\ds\de\ls
\us\ue\ls\ls\us\lu\ue\ls\us\ue\ls ;
\draw (0,1.5) 
\us\ue\ls\ls\ls\ls\ds\ld\de\ds\de 
\ds\de\ls\ls\ls\ls\us\lu\ue\us\ue ;
\draw (0,.5)  
\ls\ds\ld\de\ls\ls\ls\ls\us\lu\ue 
\ls\us\lu\ue\ls\ls\ls\ls\ds\ld\de ;
\draw (0,-.5) 
\ls\us\ue\ds\de\us\ue\us\ue\ls\ls
\ls\ds\de\us\ue\ds\de\ds\de\ls\ls ;
\draw (0,-1.5)
\ls\ls\us\lu\lu\ue\us\ue\ls\ls\ls
\ls\ls\ds\ld\ld\de\ds\de\ls\ls\ls ;
\draw[<-](0,3.5) -- (0,-2.5);
\draw[->](22,-2.5) -- (22,3.5);
\node[left] at (0,2.5) {\ft $a$};
\node[left] at (0,1.5) {\ft $b$};
\node[left] at (0,0.5) {\ft $c$};
\node[left] at (0,-.5) {\ft $d$};
\node[left] at (0,-1.5) {\ft $e$};
\node[right] at (22,2.5) {\ft $a$};
\node[right] at (22,1.5) {\ft $b$};
\node[right] at (22,0.5) {\ft $c$};
\node[right] at (22,-.5) {\ft $d$};
\node[right] at (22,-1.5) {\ft $e$};
\end{scope}
\end{tikzpicture}
\caption{\ft {\bf Top left}: A point configuration $\cal P$; {\bf Top
    right}: The dual wiring diagram $\cal W$ encoding the half-period
  of ordered switches
  $(21, 43, 53, 54, 51, 41, 52, 42, 31, 32)$; {\bf Bottom}: Extension
  to the double cover resulting in 
  the system $\cal F$ represented by smooth $2\pi$-periodic
  functions. Notice that the cyclic ordering of each triple of points
  is encoded by the order in which they appear on the upper envelope
  of the corresponding triple of curves.} 
\label{cover}
\end{figure}
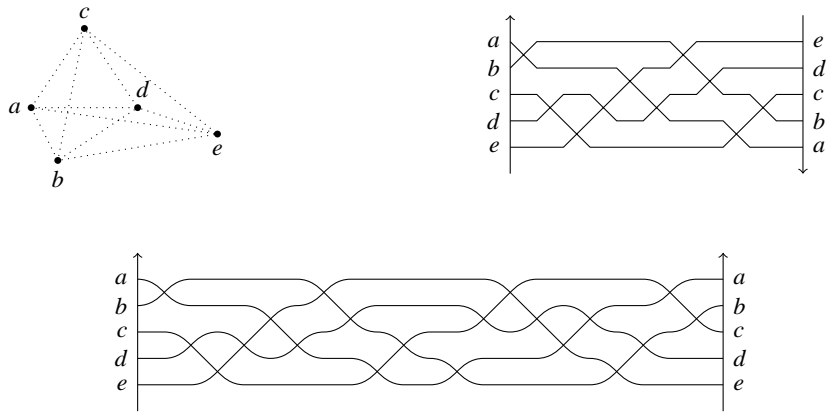

\subsection{Weak maps} Let $\cal A$ and $\cal B$ be arrangements with
$|{\cal A}| = |{\cal B}|$. A bijection $\phi: \cal A \to \cal B$ is a
\df{weak map} if $\phi^{-1}(\cal B')$ is convexly independent for
every convexly independent sub-arrangement $\cal B'\subset \cal
B$. The inequality $h_1(n)\leq g(n)$ is a consequence of the
following.  

\begin{lemma} \label{weak}
For every non-crossing arrangement $\cal A$, in which every triple is
convexly independent, there exists a weak map $\phi: \cal A \to \cal
B$ where $\cal B$ is an orientable arrangement.    
\end{lemma}

The dual arrangement $\cal A^*$ induces a cell complex $\cal C(\cal
A)$, homeomorphic to $\mb{S}^1\times [0,1]$, and the weak map $\phi$
will be defined in terms of elementary operations on $\cal C(\cal
A)$. Since every triple of $\cal A$ is convexly independent, there are
two types of triples to consider: The {\em orientable} and the {\em
  non-orientable} ones. The orientable triples were discussed in
section \ref{orientable}. The dual of a non-orientable triple $\cal T$
is characterized by one of its support curves appearing two distinct
times on the upper envelope of $\cal T^*$. In $\cal C(\cal T)$ this
corresponds to a pair of disjoint triangular cells whose top edges are
both contained in the same support curve. Notice that in the
non-orientable case these are the only triangular cells, while in the
orientable case every cell is triangular.  

An equivalent way of distinguishing the two types of triples is by
considering the cyclic order in which a curve intersects the other two
in the dual system. We call the cyclic sequence $(x,y,x,y)$
\df{alternating}, and the cyclic sequence $(x,x,y,y)$ \df{separating}.  

\begin{obs} \label{locals} Let $\cal T$ be a convexly independent
  triple of bodies with dual system ${\cal T}^*$. 
\begin{enumerate}
  \item If $\cal T$ is orientable, then for any $\gamma\in {\cal T}^*$
    the cyclic order in which $\gamma$ intersects the curves of ${\cal
      T}^*\setminus\{\gamma\}$ is alternating. 

\item If $\cal T$ is non-orientable, then for any $\gamma\in {\cal
    T}^*$ the cyclic order in which $\gamma$ intersects the curves of
  ${\cal T}^*\setminus\{\gamma\}$ is separating. 
\end{enumerate}
\end{obs}

To see why this holds we refer the reader to \fig \ref{Types}.

\begin{figure}[h!]
  \centering
\begin{tikzpicture}[scale=.36]
\begin{scope}

\begin{scope}
\fill[lightgray, opacity = .5]
(0,0) .. controls +(0:1) and +(50:1) .. 
 (2,-2) .. controls +(50:-1) and +(-5:1) ..
 (-1,-3) .. controls +(-5:-1) and +(0:-1) ..
 (0,0);
\draw[thick]
 (0,0) .. controls +(0:1) and +(50:1) .. 
 (2,-2) .. controls +(50:-1) and +(-5:1) ..
 (-1,-3) .. controls +(-5:-1) and +(0:-1) ..
 (0,0);
\end{scope}

\begin{scope}[yshift = 1cm]
\fill[lightgray, opacity = .5]
(3,0.5) .. controls +(0:1) and +(70:1) .. 
 (5,-1.5) .. controls +(70:-1) and +(-5:1) ..
 (2,-2.7) .. controls +(-5:-1) and +(80:-1) .. 
 (1,-1.5) .. controls +(80:1) and +(0:-1) ..
 (3,0.5);
\draw[thick]
 (3,0.5) .. controls +(0:1) and +(70:1) .. 
 (5,-1.5) .. controls +(70:-1) and +(-5:1) ..
 (2,-2.7) .. controls +(-5:-1) and +(80:-1) .. 
 (1,-1.5) .. controls +(80:1) and +(0:-1) ..
 (3,0.5);
\end{scope}

\begin{scope}
\path (-1,-3) ++(-5:5) coordinate (x);
\fill[lightgray, opacity = .5]
(5,0) .. controls +(0:2) and +(0:1) ..
 (x) .. controls +(0:-2) and +(0:-1) ..
 (5,0);
\draw[thick]
 (5,0) .. controls +(0:2) and +(0:1) ..
 (x) .. controls +(0:-2) and +(0:-1) ..
 (5,0);
\end{scope}

\begin{scope}
\node at (0,-1) {\ft $a$};
\node at (3,0) {\ft $b$};
\node at (4.5,-2.1) {\ft $c$};
\end{scope}

\end{scope}
\draw(-1.3,-6) \ds --++(1,-1) \de \us --++(1,1) \ue --++(1,0);
\draw(-1.3,-7) \us\ue\ds --++(1,-1)\de\us\ue;
\draw(-1.3,-8) --++(1,0) \us --++(1,1)\ue\ds --++(1,-1) \de;
\draw[->] (-1.3,-8.5)--++(0,3);
\draw[->] (5.7,-8.5)--++(0,3);

\node at (-1.9,-6) {\ft $a$};
\node at (-1.9,-7) {\ft $b$};
\node at (-1.9,-8) {\ft $c$};

\end{tikzpicture}
\hspace{2cm}
\begin{tikzpicture}[scale = .36]
  \begin{scope}

\begin{scope}
\fill[lightgray, opacity = .5]
(0,0) .. controls +(0:1) and +(50:1) .. 
 (2,-2) .. controls +(50:-1) and +(-5:1) ..
 (-1,-3) .. controls +(-5:-1) and +(0:-1) ..
 (0,0);
\draw[thick]
 (0,0) .. controls +(0:1) and +(50:1) .. 
 (2,-2) .. controls +(50:-1) and +(-5:1) ..
 (-1,-3) .. controls +(-5:-1) and +(0:-1) ..
 (0,0);
\end{scope}

\begin{scope}[xshift = -1.5cm, yshift = -1.8cm, scale = 1.2, rotate = 34]
\fill[lightgray, opacity = .5]
(3,0.5) .. controls +(0:1) and +(70:1) .. 
 (5,-1.5) .. controls +(70:-1) and +(-5:1) ..
 (2,-3.7) .. controls +(-5:-1) and +(80:-1) .. 
 (1,-1.5) .. controls +(80:1) and +(0:-1) ..
 (3,0.5);
\draw[thick]
 (3,0.5) .. controls +(0:1) and +(70:1) .. 
 (5,-1.5) .. controls +(70:-1) and +(-5:1) ..
 (2,-3.7) .. controls +(-5:-1) and +(80:-1) .. 
 (1,-1.5) .. controls +(80:1) and +(0:-1) ..
 (3,0.5);
\end{scope}

\begin{scope}
\path (-1,-3) ++(-5:5) coordinate (x);
\fill[lightgray, opacity = .5]
(5,0) .. controls +(0:2) and +(0:1) ..
 (x) .. controls +(0:-2) and +(0:-1) ..
 (5,0);
\draw[thick]
 (5,0) .. controls +(0:2) and +(0:1) ..
 (x) .. controls +(0:-2) and +(0:-1) ..
 (5,0);
\end{scope}

\node at (-.5,-1.75) {\ft $a$};
\node at (2.4,0) {\ft $b$};
\node at (5.2,-1.2) {\ft $c$};

\end{scope}
\draw(-1.3,-6) \ds --++(1,-1) \de --++(1,0) \us --++(1,1) \ue;
\draw(-1.3,-7) \us \ue \ds \du \ue\ds \de;
\draw(-1.3,-8) --++(1,0) \us --++(1,1) \ud --++(1,-1) \de --++(1,0);
\draw[->] (-1.3,-8.5)--++(0,3);
\draw[->] (5.7,-8.5)--++(0,3);

\node at (-2,-6) {\ft $a$};
\node at (-2,-7) {\ft $b$};
\node at (-2,-8) {\ft $c$};

\end{tikzpicture}
\caption{\ft The two types of triples of bodies, orientable (left) and
  non-orientable (right). Notice that the cyclic order in which we
  meet the bodies when traversing the convex hull corresponds to the
  order in which we meet the curves when transversing the upper
  envelope. In the orientable case (left) any curve meets the others
  alternatingly, for instance, the cyclic order in which $\alpha$
  meets $\beta$ and $\gamma$ is $(\gamma,\beta,\gamma,\beta)$. In the
  non-orientable case $\alpha$ meets $\beta$ and $\gamma$ in the
  cyclic order $(\beta, \gamma, \gamma, \beta)$.} 
\label{Types}
\end{figure}
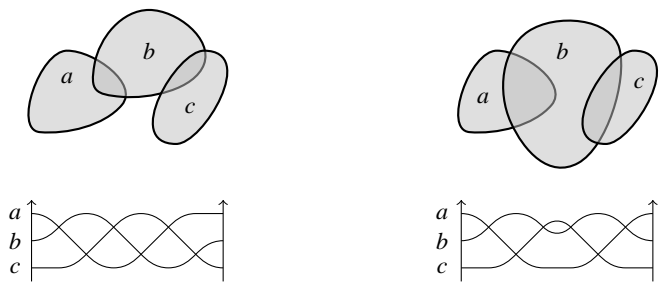

A non-orientable triple $\cal T$ is related to an orientable one by an
elementary operation called a \df{triangle flip}, which is defined by
``flipping'' the orientation of one of the two triangular cells of
$\cal C(\cal T)$. Notice that a triangle flip defines a weak map from
a non-orientable triple to an orientable one. See \fig \ref{pump1}. 

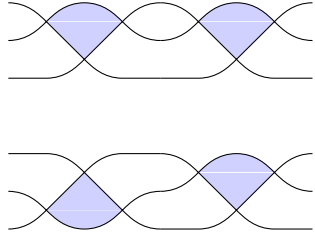
\begin{figure}[h!]
\begin{center}
\begin{tikzpicture}[scale =.5]

\begin{scope}[yshift = 3cm]
\fill[blue!60, opacity =.3]
(1,2.5)\ue\ds
(1,2.5)\ld\lu;
\draw
(0,3) \ds\ld \de \ls
(0,2) \us\ue\ds\de
(0,1) \ls\us \lu\ue ;
\end{scope}

\begin{scope}[yshift = 3cm, xshift = 4cm]
\fill[blue!60, opacity =.3]
(1,2.5)\ue\ds
(1,2.5)\ld\lu;
\draw
(0,3) \ds\ld \de \ls
(0,2) \us\ue\ds\de
(0,1) \ls\us \lu\ue ;
\end{scope}

\begin{scope}[yshift = -1cm, xshift = 4cm]
\fill[blue!60, opacity =.3]
(1,2.5)\ue\ds
(1,2.5)\ld\lu;
\draw
(0,3) \ds\ld \de \ls
(0,2) \us\ue\ds\de
(0,1) \ls\us \lu\ue ;
\end{scope}

\begin{scope}[yshift = -1cm] 
\fill[blue!60, opacity =.3]
(1,1.5)\lu\ld
(1,1.5)\de\us;
\draw
(0,1) \us\lu \ue \ls
(0,2) \ds\de\us\ue
(0,3) \ls\ds \ld\de ;
\end{scope}

\end{tikzpicture}
\caption{\ft A non-orientable triple (above) and the orientable triple
  (below) obtained after applying a triangle flip.} 
\label{pump1}
\end{center}
\end{figure}

We deduce Lemma \ref{weak} from the following.

\begin{lemma} \label{pumping} If $\cal A$ is not orientable, then
  $\cal C(\cal A)$ contains a triangular cell bounded by the support
  curves of a non-orientable triple. \end{lemma}

\begin{proof}[Proof of Lemma \ref{weak}] If $\cal A$ is not
  orientable, Lemma \ref{pumping} implies that we can apply a triangle
  flip to $\cal C(\cal A)$ obtaining a new cell complex $\cal
  C'$. Clearly we may assume $\cal C'$ is induced by a system of
  smooth curves, which is therefore the dual system of an arrangement
  $\cal A'$ (as in the proof of Lemma \ref{realz}). This induces a
  weak map $\phi' :\cal A \to \cal A'$. Since a triangle flip reduces
  the number of non-orientable triples, Lemma \ref{weak} follows by
  induction. \end{proof}  

\bigskip

A few technical terms are needed for proving Lemma \ref{pumping}. Let
$\cal T$ be a non-orientable triple. The \df{top edges} of the two
triangular cells of $\cal C(\cal T)$ belong  to the same support
curve, called the \df{top curve}, which appears twice on the upper
envelope of $\cal T^*$. When $\cal T$ belongs to a larger system $\cal
F$, the triangular cells of $\cal C(\cal T)$ may no longer be cells in
$\cal C(\cal F)$, so instead we refer to these open triangular regions
as the \df{zones} of $\cal T^*$. When we say that $\cal T^*$ bounds a
zone, it is implicit that $\cal T$ is non-orientable. A zone is called
\df{empty} if no curve of the system intersects its interior, and is
called \df{free} if no curve intersects its top edge. See \fig
\ref{fig:zones}.  

\begin{figure}[h!]
\centering
\begin{tikzpicture}
\begin{scope}[scale = .5]
\fill[blue!40, opacity= .4] 
(3, 4.5) \ue\ls\ls\ds --cycle 
(3, 4.5) \ld\ld\lu\lu --cycle;
\fill[blue!40, opacity= .4] 
(10, 3.5) \ue\ds\ld --cycle 
(10, 3.5) \ld\ld\lu --cycle;
\node [left] at (0,4) {\ft $c$};
\node [left] at (0,2) {\ft $b$};
\node [left] at (0,1) {\ft $a$};
\draw (0,5) 
\ds \ld \ld \ld \de \us \ue \ls \ls \ls \us \lu \lu \ue \ls;
\draw[blue] (0,4) 
\us \ue \ds \ld \ld \ld \de \ls \ls \ls \ls \us \lu \lu \ue;
\draw (0,3) 
\ds \de \us \lu \ue \ds \de \us \lu \ue \ls \ls \ds \ld \de;
\draw[purple] (0,2) 
\us \lu \lu \ue \ls \ls \ds \ld \de \us \ue \ds \ld \de \ls;   
\draw[blue] (0,1) 
\ls \ls \ls \us \lu \lu \lu \ue \ds \ld \ld \ld \de \ls \ls; 
\end{scope}
\end{tikzpicture}
\caption{\ft The triple $\cal T^* = \{a, b, c\}$ bounds two zones
  (shaded) and the top curve is $b$
    (red). Neither of the zones of $\cal T^*$ are empty, but the left
    one is free.}
\label{fig:zones}
\end{figure}
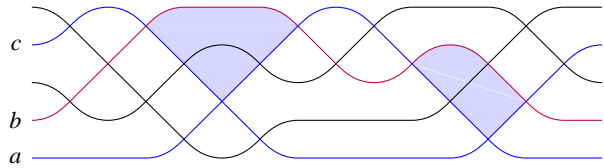

Our goal is therefore to show: \begin{center}{\em If $\cal A$ contains
    non-orientable triples, then $\cal C(\cal A)$ contains an empty
    zone.} \end{center} 

For the proof we will consider a minimal counter-example. It is,
however, easier to handle free zones rather than empty ones, so
we first establish the following.  

\begin{claim} \label{free zone}
If $\cal C(\cal A)$ contains a free zone, then $\cal C(\cal A)$
contains an empty zone.  
\end{claim}

\begin{proof} We assume without loss of generality that $Z_0$ is a
  free zone bounded by $a,b,c$ where $b$ is the top curve. Let $w_1,
  \dots, w_k$ denote curves that intersect $Z_0$. We first make some
  simple observations (these will also be of use later): 

\begin{enumerate}

\item {\em Each $w_i$ intersects $Z_0$ in a single connected arc. We
    may assume $w_i$ enters $Z_0$ by crossing curve $c$ and exits
    $Z_0$ by crossing curve $a$.} 

To see why this holds assume for contradiction that $w_i$ intersects
$Z_0$ in more than one connected arc. Up to symmetry we may then
assume $w_i$ enters $Z_0$ by crossing $c$ and then immediately crosses
$c$ again. In this case, whenever $w_i$ is above curve $c$, it is also
below curve $b$, and consequently only $b$ and $c$ appear on the upper
envelope of the triple $\{b,c,w_i\}$, which contradicts our assumption
that every triple of $\cal A$ is convexly independent.  

\item {\em The triangular region in $Z_0$ bounded by $a$, $w_i$, $c$ is a zone.}

If the region in question is not a zone, then $w_i$ should meet the
curves $a$ and $c$ alternatingly. Therefore after $w_i$ exists $Z_0$
by crossing $a$ it should meet curve $c$ before it meets curve $a$
again. Arguing as in (1), this would imply that only curves $b$ and
$c$ appear on the upper envelope of the triple $\{b,c,w_i\}$. 

\item {\em Distinct curves $w_i$ and $w_j$ cross at most {\em once}
    inside $Z_0$.}  

Suppose $w_i$ and $w_j$ are distinct curves which intersect $Z_0$ and
they cross twice inside $Z_0$. We may assume the $w_i$ enters $Z_0$
above $w_j$, which implies that $w_i$ also exists $Z_0$ above $w_j$,
and therefore the only time $w_j$ is above $w_i$ it is also below
$b$. Consequently, the upper envelope of $\{b, w_i, w_j\}$ consists
only of curves $b$ and $w_i$. 

\end{enumerate}

Of course, the zones appearing inside $Z_0$ are not necessarily free, so
we also need the following observation concerning zones that are not free.

\begin{obs} \label{splitting mutants}
Let $Z$ be a zone bounded by $a, b, c$ where $b$ is the top
curve. Suppose $w$ enters $Z$ by crossing $c$ and exits $Z$ by
crossing $b$, then proceeds to cross $a$. Then one of the triples $a,
w, b$ or $w,b,c$ bound a zone with top vertex at the crossing point
between $b$ and $w$ on the top edge of $Z$. 
\end{obs}

To see why this holds we notice that after $w$ leaves $Z$ and crosses
$a$, it enters a digon bounded by curves $a$ and $b$, which means that
the next curve that $w$ crosses must be one of these two. If $w$ exits
the digon by crossing curve $a$, then $w$ crosses curves $a$ and $b$
in a separating cyclic sequence, which implies that the triple $a$,
$w$, $b$ is non-orientable, and they therefore bound a
zone. Otherwise, $w$ exists the digon by crossing curve $b$, which
implies that $w$ crosses curves $b$ and $c$ in a separating cyclic
sequence, and consequently $w,b,c$ is non-orientable, and they bound a
zone. See \fig \ref{split zones}. 

\begin{figure}[h!]
\begin{center}
   \begin{tikzpicture}

\begin{scope}[scale = .4, yshift = 7cm]
\fill[blue!40, opacity= .4] 
(2, 3.5) \lu\ue\ds\ld --cycle 
(2, 3.5) \ld\ld\lu\lu --cycle;
\node [left] at (0,5) {\ft $c$};
\node [left] at (0,3) {\ft $b$};
\node [left, red] at (0,2) {\ft $w$};
\node [left] at (0,1) {\ft $a$};
\draw(0,5) 
\ds \ld \ld \ld \de \ls \ls \ls \ls \ls \ls;
\draw (0,3) 
\ls \us \lu \ue \ds \ld \ld \de \us \lu \ue;
\draw[red,->, >=latex] (0,2) 
\ls \ls \us \lu \lu \ue \ds --++ (.5,-.5);   
\draw(0,1) 
\ls \ls \ls \us \lu \lu \lu \ue \ds \ld \de; 
\end{scope}

\begin{scope}[xshift=3cm, scale=.4]
\fill[blue!40, opacity= .4] 
(2, 3.5) \lu\ue\ds --cycle 
(2, 3.5) \ld\lu\lu --cycle;
\node [left] at (0,5) {\ft $c$};
\node [left] at (0,3) {\ft $b$};
\node [left, red] at (0,2) {\ft $w$};
\node [left] at (0,1) {\ft $a$};
\draw(0,5)
\ds \ld \ld \ld \de \ls \ls \ls \ls \ls \ls;
\draw (0,3) 
\ls \us \lu \ue \ds \ld \ld \de \us \lu \ue;
\draw[red, ->, >=latex] (0,2) 
\ls \ls \us \lu \lu \ue \ds \ld \ld \de ;   
\draw(0,1) 
\ls \ls \ls \us \lu \lu \lu \ue \ds \ld \de; 
\end{scope}

\begin{scope}[xshift=-3cm, scale=.4]
\fill[blue!40, opacity= .4] 
(5, 4.5) \ue\ds --cycle 
(5, 4.5) \ld\lu --cycle;
\node [left] at (0,5) {\ft $c$};
\node [left] at (0,3) {\ft $b$};
\node [left, red] at (0,2) {\ft $w$};
\node [left] at (0,1) {\ft $a$};
\draw(0,5) 
\ds \ld \ld \ld \de \ls \ls \ls \ls \ls \ls;
\draw (0,3)  
\ls \us \lu \ue \ds \ld \ld \de \us \lu \ue;
\draw[red, ->, >=latex] (0,2) 
\ls \ls \us \lu \lu \ue \ds \de \us \ue;   
\draw(0,1)
\ls \ls \ls \us \lu \lu \lu \ue \ds \ld \de; 
\end{scope}

\end{tikzpicture}
\caption{\ft {\bf Top:} The zone $Z$ is bounded by $a,b,c$ (shaded).
After $w$ leaves $Z$ and crosses $a$ it
enters a digon bounded by curves $a$ and $b$, so it must cross one of
them again before crossing $c$. {\bf Bottom left:} If the
next crossing of $w$ is with $a$, then  $a, w, b$ bound a zone (shaded). {\bf
  Bottom right:} If the next crossing of $w$ is with 
$b$ then $w,b,c$ bound a zone (shaded).}
\label{split zones}
\end{center}
\end{figure}
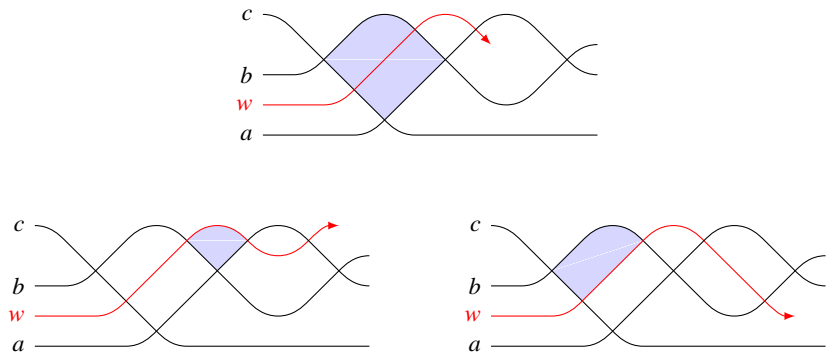

For the proof of Claim \ref{free zone} we proceed by induction on $k$,
the number of curves which intersect $Z_0$. If $k=0$, then $Z_0$ is an
empty zone, so assume $k>0$. Start at the top left corner of $Z_0$ at
the crossing between $b$ and $c$. Move on the boundary of $Z_0$ along
$c$ and stop at the first crossing we encounter. Assume that this is
the crossing between $c$ and $w_k$. This crossing is the top corner of
a zone $Z_k \subset Z_0$ (bounded by $a, w_k,c$ by (2) above). Move
into the interior of $Z_0$ along curve $w_k$ (the top edge of $Z_k$)
and stop at the first crossing we encounter. If this is the crossing
between $w_k$ and $a$, we may apply the induction step, since then
$Z_k$ is a free zone with less than $k$ intersecting curves. So assume
that this is a crossing between $w_k$ and $w_{k-1}$. By Observation
\ref{splitting mutants} (with $b = w_k$ and $w = w_{k-1}$) this
crossing is the top vertex of a zone $Z_{k-1}$. If $Z_{k-1}\subset
Z_k$ (i.e bounded by $w_{k-1}, w_k, c$), then $Z_{k-1}$ is free and we
are done by induction, so assume $Z_{k-1}$ is bounded by curves $a$,
$w_{k-1}$, $w_k$. Now proceed along curve $w_{k-1}$ (the top edge of
$Z_{k-1}$) and repeat the process. In general, we proceed from the
left top vertex of the zone $Z_j$ along the curve $w_{j}$ (the top
edge of $Z_j$) and stop at the first curve we meet. If the first
curve we meet is $a$, then we are done because $Z_j$ is a free zone
intersected by fewer than $k$ curves. Otherwise we meet curve
$w_{j-1}$, and by Observation \ref{splitting mutants} this crossing
point is the top vertex of a zone $Z_{j-1}$. If $Z_{j-1}\subset Z_j$
then we are done, or else we can repeat by proceeding along $w_{j-1}$
(the top edge of $Z_j$). This process simply amounts to moving along
the upper envelope of the curves $\{a, c, w_1, \dots, w_k\}$ within
the zone $Z$. By (3) above, each pair $w_i$ and $w_j$ cross at most
once within $Z$, so each curve can appear on the upper envelope at
most once (inside $Z$). Since there are only finitely many curves
$w_i$ the process must eventually end, either with a free zone $Z_i
\subset Z_{i+1}$, or with the first curve which $w_i$ meets after
appearing on the upper envelope being $a$, in which case $w_{i+1},
w_i, a$ bound a free zone $Z_i \subset Z$. In either case we
eventually reach a free zone which is crossed by less than $k$ curves,
completing the proof of Claim \ref{free zone}. See \fig \ref{find
  zone}. \end{proof}  

\begin{figure}[h!]
\centering
\begin{tikzpicture}
\begin{scope}[scale =.45, yscale=1.2]
\node [left] at (0,7) {\ft $c$};
\node [left] at (0,6) {\ft $b$};
\node [left] at (0,5) {\ft $w_k$};
\node [left] at (0,3) {\ft $w_{k-1}$};
\node [left] at (0,1) {\ft $a$};
\fill[blue!30, opacity= .4] 
(1, 6.5) \ue \ls \ls \ls \ls \ls \ls \ls \ls\ds  --cycle 
(1, 6.5) \ld\ld \ld \ld \ld \lu\lu\lu \lu \lu--cycle;
\fill[blue!60, opacity= .4] 
(2, 5.5) \ue \ls \ls \ds  --cycle 
(2, 5.5) \ld\ld \lu\lu --cycle
(6, 5.5) \ue \ls \ls \ds  --cycle 
(6, 5.5) \de \ls \ds \lu --cycle;
\draw(0,7) 
\ds\ld\ld\ld\ld \ld \de \ls\ls\ls\ls \ls;
\draw(0,6)
\us \ue \ls\ls\ls \ls \ls \ls \ls \ls \ds \de;
\draw(0,5)
\ls\us \ue \ls\ls \ds\de \ls \ds \de ;
\draw(0,4.25)
\ls \ls \us \ue \ds --++ (.5,-.5);
\draw(0,4)
\ls \ls \us \ue \ds --++ (.4,-.4);
\draw(0,3.75)
\ls \ls \us \ue \ds --++ (.3,-.3);
\draw(0,3)
\ls \ls \ls \us \lu \lu \ue \ls \ls \ds \de;
\draw(0,2.25)
\ls \ls \ls \ls \us --++ (.3,.3);
\draw(0,2)
\ls \ls \ls \ls \us --++ (.4,.4);
\draw(0,1.75)
\ls \ls \ls \ls \us --++ (.5,.5);
\draw(0,1)
\ls \ls \ls \ls \ls \us \lu \lu \lu \lu \lu \ue;
\draw[red, thick, ->, >=latex] (1,6.5)\ld \ue \ls \ls \ds ;
\fill[red] (1,6.5) circle [radius = 0.08];
\end{scope}
\end{tikzpicture}
\caption{\ft Starting at top left corner of $Z_0$ (light shade) move along the
  boundary until we meet the first crossing. This is the top corner
  of a zone bounded by $a, w_k, c$. Proceed along $w_k$
  until we meet the next crossing. By Observation \ref{splitting
    mutants} one of the two dark shaded regions must be a zone.}
\label{find zone}
\end{figure}
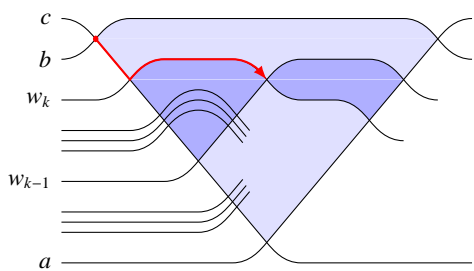

We are in position to complete the proof of Lemma \ref{pumping}.

\begin{proof}[Proof of Lemma \ref{pumping}] Suppose $\cal A$ is a
  minimal counter-example. Then $\cal C(\cal A)$ contains zones, but
  no empty ones, and any proper sub-arrangement $\cal A'\subset
  \cal A$ is either orientable or the complex $\cal C(\cal A')$ has at
  least one empty zone. We will reach a contradiction by showing that
  $\cal C(\cal A)$ contains a {\em free zone}.  

Assume first that {\em any} curve we delete from the {\em lower
  envelope} of $\cal A^*$ (defined as the graph of the function
$\min_i \gamma_i$) 
destroys all non-orientable triples. Then the lower envelope consists
of exactly {\em two} curves $a$ and $c$, and a triple
is non-orientable if and only if it includes both of these curves. To
see this, note that if there were three curves on the lower envelope,
then these form an orientable triple, so for any non-orientable
triple there is a curve on the lower envelope not belonging to it.
Let $Z$ be a zone bounded by $a$, $b$, $c$. Some curve $w$ should
intersect the top edge of $Z$ or else it is free, in which case $a$,
$b$, $w$ or $b$, $c$, $w$ is non-orientable, contradicting our
initial assumption. 

We may therefore assume that there is a curve $w$ appearing on the lower
envelope, and a triple $a,b,c$ which bound a zone $Z$ where $b$ is the
top curve, and $w$ is the {\em only} curve which meets the interior of
$Z$. Furthermore $w$ must cross the top edge of $Z$ (if not $Z$ is
free, contradicting the assumption that $\cal A$ was a counter-example). 
Now we use the fact that $w$ was on
the lower envelope. This implies that $w$ must also cross one of the
other edges of $Z$. Up to symmetry there are then two cases that can
occur (see \fig \ref{mini-count}).

\begin{enumerate}
\item $w$ is on the lower envelope, crosses $a$, then $c$ (entering
  $Z$), and then $b$ (leaving $Z$). 
\item $w$ is on the lower envelope, crosses $c$, then $a$ (entering
  $Z$), and then $b$ (leaving $Z$). 
\end{enumerate}

In both cases we consider the order in which $w$ intersects the other
curves after leaving $Z$. 

In case (1) it results in a zone bounded by $w, b, c$ contained in
$Z$. This must be empty, since $w$ is the only curve that intersects
$Z$. We use Observation \ref{locals} to argue that $w, b, c$ come from
a non-orientable triple. Notice that after $w$ leaves $Z$ by crossing
$b$ it must cross $b$ again before crossing $c$, or else only curves
$c$ and $w$ appear on the upper envelope of the triple $a, w,
c$. Therefore the cyclic order in which $w$ intersects curves $b$ and
$c$ is separating.

In case (2) this results in a zone bounded by $w, a, b$ where $a$ is
the top curve. It is adjacent to $Z$ along the curve $a$, which
implies that it is a free zone since $w$ is the only curve
intersecting the interior of $Z$. Again we use Observation
\ref{locals} to argue that $a, w, c$ come from a non-orientable
triple. This depends on the order in which curve $w$ meets curves $a$
and $c$ after leaving $Z$. However, if $w$ intersects $c$ before $a$,
then only curves $b$ and $w$ appear on the upper envelope of the
triple $b, w, c$. 

We can therefore conclude that there always exists an free zone. This
contradicts our assumption that this was a minimal counter-example,
which completes the proof. \end{proof}   

\begin{figure}[h!]
\begin{center}
 \begin{tikzpicture}
\begin{scope}[scale = 0.5]
\fill[blue!40, fill opacity=0.4] 
(1,2.5) \ue \ds -- cycle
(1,2.5) \ld \lu;
\draw (0,3)node[left]{\ft $c$} 
(0,3) \ds \ld \ld \de \ls \ls \ls; 
\draw (0,2)node[left]{\ft $b$} 
(0,2) \us \ue \ds \ld \de \us \ue; 
\draw (0,1)node[left]{\ft $a$} 
(0,1) \ds \de \us \lu \ue \ds \de;
\draw[purple, ->, >=latex] (0,0)node[left]{\ft $w$} 
(0,0) \us \lu \lu \ue; 
\end{scope} 

\begin{scope}[scale=0.5, xshift = 12cm]
\fill[blue!30, fill opacity=0.6] 
(2,1.5) \ue \ds --cycle
(2,1.5) \ld \lu --cycle;

\draw (0,3)node[left]{\ft $c$}  
(0,3) \ds \ld \ld \de \ls \ls \ls \ls \ls;
\draw (0,2)node[left]{\ft $b$} 
(0,2) \us \ue \ls \ls \ds \ld \de \us \ue; 
\draw (0,1)node[left]{\ft $a$} 
(0,1) \ls \us \ue \ds \de \us \ue \ds \de; 
\draw[purple, ->, >=latex] 
(0,0)node[left]{\ft $w$} 
(0,0) \ls \ls \us \lu \lu \ue; 
\end{scope}
\end{tikzpicture}
\caption{\ft Consider $w$ after it leaves $Z$. {\bf Case (1), left:}  If
  $w$ crosses $b$ before $c$, then $w, b, c$ bound an empty zone
  contained in $Z$. If $w$ crosses $c$ before $b$, then $w,a,c$ is not
  in convex position. {\bf Case (2), right:} If $w$ crosses $a$ before
  $c$, then $w, a, c$ bound a free zone below $Z$. If $w$ crosses $c$
  before $a$, then $b$ intersects $w$ again after its two crossings
  with $a$, which implies that $w, a, b$ do not come from a convexly
  independent triple.} 
\label{mini-count}
\end{center}
\end{figure}
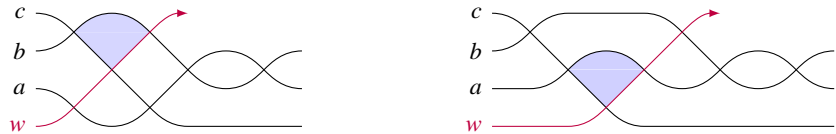

\section{The upper bound for generalized
  configurations} \label{uppers}  

Here we show that $g(n) \leq \binom{2n-5}{n-2}+1$ for all $n\geq
5$. The proof is dual to the proof given by T\'{o}th and Valtr in
\cite{totval}, but more general since it is in terms of
arbitrary wiring diagrams. 

There are two basic wiring diagrams which we refer to throughout. A
\df{$k$-cup} is the unique wiring diagram on $k$ wires in which every
wire appears on the upper envelope, and an \df{$l$-cap} is the unique
wiring diagram on $l$ wires in which every wire appears on the lower
envelope. See \fig \ref{cupfig}. 

\begin{figure}[h!]
\centering
\begin{tikzpicture}
  \begin{scope}[scale = .37]
    \begin{scope}
      \draw 
      (0,2)\ds\ld\ld\ld\de\ls\ls\ls
      (0,1)\us\ue\ds\ld\ld\de\ls\ls
      (0,0)\ls\us\lu\ue\ds\ld\de\ls
      (0,-1)\ls\ls\us\lu\lu\ue\ds\de
      (0,-2)\ls\ls\ls\us\lu\lu\lu\ue;
    \end{scope}
    \begin{scope}[xshift = 15cm, yscale = -1 ]
      \draw
      (0,2)\ds\ld\ld\ld\de\ls\ls\ls
      (0,1)\us\ue\ds\ld\ld\de\ls\ls
      (0,0)\ls\us\lu\ue\ds\ld\de\ls
      (0,-1)\ls\ls\us\lu\lu\ue\ds\de
      (0,-2)\ls\ls\ls\us\lu\lu\lu\ue;
    \end{scope}
  \end{scope}
\end{tikzpicture}
\caption{\ft A 5-cup (left) and a 5-cap (right).}
\label{cupfig}
\end{figure}
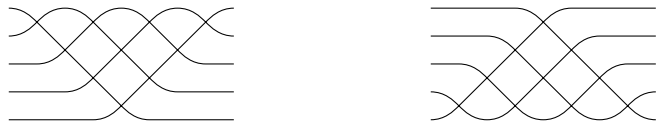

\begin{obs}\label{cupsncaps0} Let $\cal P$ be a generalized
  configuration and $W$ the wiring diagram corresponding to any
  half-period of the allowable sequence of $\cal P$. If $W$ contains a
  $k$-cup (or $k$-cap), then the wires of the $k$-cup (or $k$-cap)
  correspond to a $k$-tuple of $\cal P$ which is convexly independent.  
\end{obs}

It is important to note that the converse of Observation
\ref{cupsncaps0} does not hold, that is, $\cal P$ may have $k$
convexly independent points, but no half-period of the allowable
sequence contains a $k$-cup or $k$-cap. Also notice that every wiring
diagram on 3 wires is a 3-cup or a 3-cap. In the following observation
we assume that the wires are ordered by the way in which they
intersect a vertical line at the start (left) of the half-period. 

\begin{obs}\label{cupsncaps1}
  Suppose $W$ is a $k$-cup and $W'$ is an $l$-cap where $x$ is the
  {\em bottom} wire of $W$ and the {\em top} wire of $W'$. Then $W\cup W'$
  contains a $(k+1)$-cup or an $(l+1)$-cap.  
\end{obs}

This observation is dual to the one appearing in \cite{erd-sze1}. It
also holds if we switch the roles of the top and bottom wires. To see
why the observation holds, consider the order in which the wire $x$
intersects the other wires of $W$ and of $W'$. If $x$ intersects
every wire of $W$ before intersecting every wire of $W'$, then $W$ can
be extended to a $(k+1)$-cup. Otherwise, $W'$ can be extended to an
$(l+1)$-cap.  

\begin{obs}\label{cupsncaps2}
Every wiring diagram on $\binom{k+l-4}{k-2} + 1$ wires contains a
$k$-cup or an $l$-cap.  
\end{obs}

The proof is the same as the one given by Erd\H{o}s and Szekeres
\cite{erd-sze1}. One proceeds by induction on $k+l$ and define $A$ to
be the set wires which are the bottom wire of some $(k-1)$-cup. By
induction, it follows that $|A| \geq \binom{k+l-5}{k-3}+1$, thus $A$
contains a $k$-cup or an $(l-1)$-cap. In the latter case, Observation
\ref{cupsncaps1} implies the existence of a $k$-cup or an $l$-cap. 

We now formulate a crucial Lemma which lies at the heart of the proof of
T\'{o}th and Valtr \cite{totval}. 

\begin{lemma}\label{toth-improv}
Let $\cal P$ be a generalized configuration with $|\cal P| =
\binom{2n-5}{n-2}$ and $n\geq 5$. Let $W$ be a wiring diagram corresponding to
any half-period of the allowable sequence of $\cal P$. If no subset of
$n$ points of $\cal P$ are convexly independent, then $W$ contains an
$(n-1)$-cup.   
\end{lemma}

\begin{remark}
  Notice that if $|\cal P| = \binom{2n-5}{n-2}+1$, then the conclusion
  of Lemma \ref{toth-improv} is guaranteed by Observation
  \ref{cupsncaps2}, thus the extra effort which goes into proving
  Lemma \ref{toth-improv} only pays off an improvement of a single point! 
\end{remark}

Before getting to the proof of Lemma \ref{toth-improv}, let us see why
this implies $g(n)\leq \binom{2n-5}{n-2}+1$. Let 
$\cal P$ be a generalized configuration with $|\cal P| =
\binom{2n-5}{n-2}+1$ and let $W$ be the wiring diagram of one of the
half-periods of its allowable sequence. We insert a new wire $x$
starting above every wire of $W$ such that it first crosses the
initial top wire $w$ of $W$ while $w$ is on the upper envelope, and
then follows wire $w$, slightly below, meeting the remaining wires of
$W$ in the same order as $w$ does. In this way $x$ crosses every wire
of $W$ precisely once. 

The wire $x$ separates the crossings of $W$ into ``left crossings''
and ``right crossings'', and the idea is to shift the left crossings
over to the right side by sending them through the line at
infinity. More precisely, consider the first crossing on the left side
of $x$ representing the ordered switch $ij$. At the right end of the
wiring diagram, these wires will be adjacent, but in the opposite
order. Delete the $ij$ crossing from the left side and add the
corresponding $ji$ crossing on the right side at the end of the wiring
diagram. The procedure can be repeated until all left crossings have
been evacuated, after which we can apply a homeomorphism of the plane
which maps the wire $x$ to a vertical line. We now delete wire $x$ and
obtain a new wiring diagram $W'$ on the same number of wires as $W$,
with the property that the top wire of $W'$ crosses every wire below
before any other crossings occur. See \fig \ref{evacuation}.  

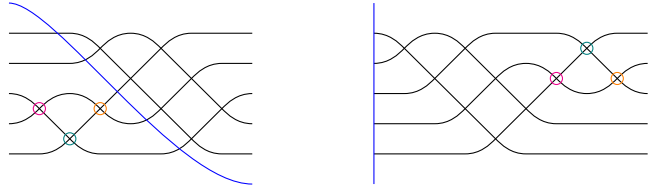
\begin{figure}[h!]
  \centering
  \begin{tikzpicture}
    \begin{scope}[scale =.4]

      \begin{scope}
        \draw (0,4) 
        \ls\ls\ds\ld\ld\ld\de\ls;
        \draw (0,3) 
        \ls\ls\us\ue\ds\ld\ld\de;
        \draw (0,2)  
        \ds\ld\de\ls\ls\us\lu\ue; 
        \draw (0,1) 
        \us\ue\ds\de\us\lu\ue\ls;
        \draw (0,0)
        \ls\us\lu\lu\lu\ue\ls\ls;
        \draw[blue] 
        (0,5) .. controls (1.5,5) and (5.5,-1) .. (8,-1);
        \draw [magenta] (1,1.5) circle [radius = .2cm];
        \draw [teal] (2,.5) circle [radius = .2cm];
        \draw [orange] (3,1.5) circle [radius = .2cm];
      \end{scope}

      \begin{scope}[xshift=12cm]
        \draw (0,4) 
        \ds\ld\ld\ld\de\ls\ls\ls\ls;
        \draw (0,3) 
        \us\ue\ds\ld\ld\de\ls\ls\ls;
        \draw (0,2)  
        \ls\us\lu\ue\ls\ls\ds\ld\de; 
        \draw (0,1) 
        \ls\ls\us\lu\ue\ds\de\us\ue;
        \draw (0,0)
        \ls\ls\ls\us\lu\lu\lu\ue\ls;
        \draw[blue] 
        (0,-1) -- (0,5);
        \draw [magenta] (6,2.5) circle [radius = .2cm];
        \draw [teal] (7,3.5) circle [radius = .2cm];
        \draw [orange] (8,2.5) circle [radius = .2cm];
      \end{scope}
    \end{scope}
  \end{tikzpicture}
  \caption{evacuation}
  \label{evacuation}
\end{figure}

It is easily verified that the above operation does not change the
orientation of any triples. That is, when we extend $W'$ to its double 
cover, we obtain the full period of the allowable sequence of a
generalized configuration $\cal P'$, and there will be a bijection
$\phi: \cal P \to \cal P'$ which preserves the cyclic order of every
triple of points. In particular, a subset of $\cal P$ is convexly
independent if and only if its image under $\phi$ is convexly
independent. (Recall the proof of Lemma \ref{realz}.)

We may therefore assume that the top wire of $W$ intersects every
other wire of $W$ before any other crossing occurs, and then consider
the sub diagram $W_0 \subset W$ obtained by deleting the top wire from
$W$. Now $W_0$ corresponds to the half-period of a subconfiguration
$\cal P_0 \subset \cal P$ consisting of $\binom{2n-5}{n-2}$ points. If
$\cal P_0$ contains $n$ points which are convexly independent, then we
are done. Otherwise, Lemma \ref{toth-improv} implies that $W_0$
contains an $(n-1)$-cup, which together with the wire we deleted from
$W$, forms an $n$-cup.   

\begin{proof}[Proof of Lemma \ref{toth-improv}] Let $\cal P$ be a
  generalized configuration on $\binom{2n-5}{n-2}$ points which does
  not contain any convexly independent subset of size $n$. Fix a
  half-period of the allowable sequence of $\cal P$ and let $W$ be its
  wiring diagram. By assumption $W$ contains no $n$-cap, and suppose
  for contradiction that $W$ contains no $(n-1)$-cup.  

  Let $A$ denote the set of wires which are the bottom wire of some
  $(n-2)$-cup and let $B$ denote the remaining wires. If $|A| >
  \binom{2n-6}{n-3}$, then Observation \ref{cupsncaps2} implies that
  $A$ contains an $(n-1)$-cap, since we've assumed that there is no
  $(n-1)$-cup. But then Observation \ref{cupsncaps1} implies that $A$
  contains an $(n-1)$-cup or an $n$-cap. Therefore $|A|\leq
  \binom{2n-6}{n-3}$. On the other hand, if $|B| > \binom{2n-6}{n-2}$,
  then $B$ contains an $(n-2)$-cup or an $n$-cap, by Observation
  \ref{cupsncaps2}, which contradicts the definition of $A$. We
  conclude that $|A| = \binom{2n-6}{n-3}$ and $|B| =
  \binom{2n-6}{n-2}$.  

  Next, we claim that for every $a\in A$, the set $B\cup \{a\}$
  contains an $(n-2)$-cup whose bottom wire is $a$. This follows from
  Observation \ref{cupsncaps2} and the definition of $A$. Similarly, we
  claim that for every $b\in B$ the set $A\cup \{b\}$ contains an
  $(n-1)$-cap whose top wire is $b$. This follows from Observation
  \ref{cupsncaps2} and by noticing that if $A$ contains an $(n-1)$-cap,
  then Observation \ref{cupsncaps1} would imply the existence of an
  $(n-1)$-cup or an $n$-cap.

  Now consider the set of pairs $(a,b)$ with $a\in A$ and $b\in B$
  such that $a$ is the bottom wire and $b$ the top wire of an
  $(n-2)$-cup or an $(n-1)$-cap, and choose the pair $(a,b)$ whose
  crossing is the leftmost among all such pairs. We now deal with two
  separate cases depending on whether the pair $(a,b)$ bound an
  $(n-2)$-cup or an $(n-1)$-cap. 

  For the first case, suppose $(a,b)$ bound an $(n-2)$-cup whose wires
  appear on the upper envelope in order $b, w, \dots, a$. By the
  argument above, there exists an $(n-1)$-cap whose wires appear on
  the lower envelope in order $a', w', \dots, b$. Note that $a\neq a'$
  or else $\cal P$ contains convexly independent subset of size
  $2n-3$, so by our assumption $b$ meets $a$ before $a'$. To see this, 
  extend to the full period and observe that every wire of the
  $(n-2)$-cup {\em and} the $(n-1)$-cap appear on the upper
  envelope. Next, observe that the starting point of $a'$ is between
  the starting points of $b$ and $w$. If not, then $a'$ meets $b$
  before $w$, which implies that the wires of the $(n-1)$-cap, $a',w',
  \dots, b$ together with $w$, correspond to a convexly independent
  subset of $\cal P$, which can be seen, again, by extending to the double
  cover. Finally, consider the order in which $w'$ meets wires $a'$
  and $w$. If $w'$ meets $a'$ before $w$, then there is an $(n-1)$-cup
  whose wires appear in order $w', a', w, \dots, a$. Otherwise, there
  is an $n$-cap whose wires appear in the order $w, a, w', \dots,
  b$. See \fig \ref{tilf}, left.  

  The second case is similar. Suppose $(a,b)$ bound an $(n-1)$-cap
  whose wires appear on the lower envelope in order $a, \dots, w,
  b$. There exists an $(n-2)$-cup whose wires appear on the upper
  envelope in order $b', \dots, w', a$. As before, we must have $b\neq
  b'$, so by our assumption $a$ meets $b$ before $b'$. Moreover, the
  starting points of $b'$ must be above the starting point of $b$, or
  else the wires of the $(n-1)$-cap $a, \dots, w, b$ together with
  $b'$ correspond to a convexly independent subset of $\cal P$ (again,
  extend to the double cover). Finally, consider the order in which
  $w'$ meets wires $b$ and $w$. If $w'$ meets $b$ before $w$, then
  there is an $(n-1)$-cup whose wires appear in the order $b, \dots,
  w', b, w$. Otherwise, there is an $n$-cap whose wires appear in
  order $a, \dots, w, b, w'$. See \fig \ref{tilf}, right. \end{proof} 

  \begin{figure}[h!]
    \centering
    \begin{tikzpicture}
      \begin{scope}[scale=.5]
        \draw(0,6) \ls\ls\ls\ds\ld\ld\ld\de;
        \draw[blue, dashed] (0,5) .. controls (2,5) and (2,3) .. (2.5,2.7); 
        \draw[blue, dashed] (0,5) .. controls (3.5,5) and (3.6,3) .. (3.7,2.6);
        \draw(0,4) \ls\ls\ds\ld\de\ls\us\lu;
        \draw(0,3) \ls\ls\us\lu\lu\ud;
        \draw(6.3,5.8)[dotted] .. controls (6.5,6) and (6.7,6) .. (6.9,5.8);
        \draw(0,2) \ls\ls\ls\us\lu\lu\lu\ue;
        \node at (-.6,2)  {\ft $a$};
        \node at (-.5,4)  {\ft $a'$};
        \node at (-.6,3)  {\ft $w$};
        \node at (-.5,5)  {\ft $w'$};
        \node at (-.6,6)  {\ft $b$};
      \end{scope}
 
      \begin{scope}[scale=.5, xshift =14cm]
        \draw(0,6) \ls\ds\ld\ld\ld\de;
        \draw(0,5) \ls\us\ue\ds--++(.5,-.5);
        \draw[dotted] (1.65,6) .. controls (1.85,6.2) and (2.05,6.2)
        .. (2.25,6);
         \draw[dotted] (0.65,2) .. controls (0.85,1.8) and (1.05,1.8)
         .. (1.25,2); 
        \draw[blue, dashed] (4.5,5) .. controls (5,4.5) and (5,3)
        .. (6,2.5) ;
        \draw[blue, dashed] (4.5,5) .. controls (5,4.5) and
        (6.5,4.5) .. (7,2.5); 
        \draw(0,4) \ls\ds\de\us\ue\ds\de;
        \draw(0,3) \ds\de\ls\ls\us\lu\ue;
        \draw(0,2) \us\lu\lu\lu\ue;
        \node at (-.6,2)  {\ft $a$};
        \node at (-.6,4)  {\ft $b$};
        \node at (-.6,3)  {\ft $w$};
        \node at (-.5,5)  {\ft $w'$};
        \node at (-.5,6)  {\ft $b'$};
      \end{scope}
    \end{tikzpicture}
    \caption{\ft First case (left): $(a,b)$ bound an $(n -2)$-cup whose wires
    appear in order $b, w, \dots, a$. Second case (right): $(a,b)$
    bound an $(n -1)$-cap whose wires appear in order $a, \dots w, b$.}
    \label{tilf}
  \end{figure}
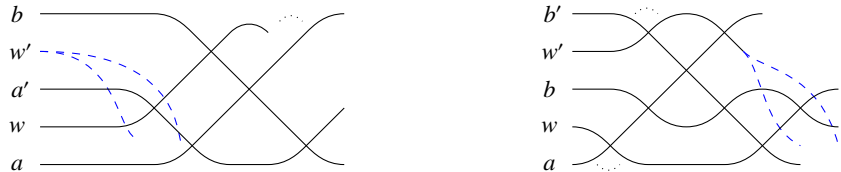

\section{Further generalizations} \label{extensions}

A \df{convex $n$-clustering} is a disjoint
union of point sets $\cal
S_1$, $\cal S_2$, $\dots, \cal S_n$ of  
equal size such that all $n$-tuples $(p_1, p_2, \dots, p_n)$ with
$p_i\in \cal S_i$ are convexly independent. The cardinality of the
$\cal S_i$ is the \df{size} of the clustering. This notion
naturally extends to arrangements of bodies as well. 

\subsection{The positive fraction version}

B\'{a}r\'{a}ny and Valtr gave the following generalization of the
Erd\H{o}s-Szekeres theorem, known as the {\em positive fraction
  Erd\H{o}s-Szekeres theorem}.   
 
\begin{unthm}[B\'{a}r\'{a}ny-Valtr \cite{baranyES}] 
  For every integer $n>3$ there exists a constant $c_n>0$ such that any
  finite set $\cal S$ in the Euclidean plane, in which every
  triple is convexly independent, contains a convex $n$-clustering of
  size $c_n |\cal S|$. 
\end{unthm}

The current best value for $c_n$ is due to P\'{o}r and Valtr
\cite{partitionES}, and shows that $c_n \geq n\cdot 2^{-32n}$. Their
argument can be repeated verbatim to hold for generalized
configurations as well. 
By Lemmas \ref{realz} and \ref{weak} the  positive fraction
version and for generalized 
configurations is equivalent 
to the positive fraction version for non-crossing arrangements of bodies, and
therefore holds with the same bound on $c_n$.\footnote{The positive
  fraction Erd\H{o}s-Szekeres theorem was first established for
  arrangements of mutually disjoint bodies by Pach and Solymosi
  \cite{pach-soly1}, and their method was subsequently improved by P\'{o}r
  and  Valtr \cite{por-convex}. Our methods imply a
  substantial quantitative improvement, as well as relaxing the
  disjointness assumption.}

\begin{theorem} 
  For every integer $n>3$ there exists a constant $c_n>0$ such that any
  non-crossing arrangement $\cal A$, in which every
  triple is convexly independent, contains a convex $n$-clustering of
  size $c_n|\cal A|$. 
\end{theorem}

\subsection{The partitioned version}

Answering a question of Kalai, the positive fraction version was
further generalized by P\'{o}r and Valtr \cite{partitionES} with what
is called the {\em partitioned Erd\H{o}s-Szekeres theorem}.  

\begin{unthm}[P\'{o}r-Valtr \cite{partitionES}]
  For every $n\geq 3$ there exist constants $p = p_n$ and $r = r_n$
  such that for any finite set $\cal S$ the Euclidean plane, in which
  every triple is convexly independent, 
  there is a sub-arrangement $\cal S'$ of 
  size at most $r$ such that $\cal S \setminus \cal S'$ can be partitioned
  into at most $p$ convex $n$-clusterings.
\end{unthm}

Extending this theorem to generalized configurations can be
done in a more or less a routine way. Essentially one needs to modify
the proofs of
Claims 1 -- 3 in \cite{partitionES}. This can be done by replacing any
``distance 
arguments'' by ``continuous sweep arguments'' (see for instance
\cite{smoro, topaff}). The remaining parts of the proof of P\'{o}r and
Valtr are combinatorial, and do not need further modification. By
applying Lemmas \ref{realz} and \ref{weak} it follows that the
fractional versions for generalized configurations and for
non-crossing arrangements are equivalent. We therefore
obtain the following.\footnote{This
  theorem was announced by P\'{o}r and Valtr in \cite{partitionES} for
  the case of pairwise disjoint bodies, 
  but their proof is complicated and appears only in an unpublished
  manuscript.} 

\begin{theorem} \label{c3part}
  For every $n\geq 3$ there exist constants $p = p_n$ and $r = r_n$
  such that the following holds. For every non-crossing
  arrangement $\cal A$ in which every triple is convexly independent
  there is a sub-arrangement $\cal A'$ of 
  size at most $r$ such that $\cal A \setminus \cal A'$ can be partitioned
  into at most $p$ convex $n$-clusterings.
\end{theorem}

\begin{remark}
It is natural to ask whether the non-crossing condition can be further
relaxed. In a subsequent paper we show that this is indeed the case,
confirming a conjecture of Pach and T\'{o}th \cite{PachToth1}.
\end{remark}

\section{Acknowledgments} We are grateful to the anonymous referee for
many useful comments which helped improve the exposition of this paper.

M.~Dobbins was supported by NRF grant 2011-0030044 funded by the
government of South Korea (SRC-GAIA) and BK21.  

A.~Holmsen was supported  by Basic Science Research Program through the
National Research Foundation of Korea funded by the Ministry of
Education, Science and Technology (NRF-2010-0021048). 

A.~Hubard was supported
by Fondation Sciences Math\'{e}matiques de Paris. A.~Hubard would like to thank
KAIST for their hospitality and support during his visit.


\end{document}